\documentclass[12pt,reqno]{amsart}
\usepackage{amsmath}
\usepackage{amsthm}
\usepackage{amssymb}
\usepackage{pst-vue3d}
\usepackage{multido}
\usepackage[dvips]{graphics}
\large\normalsize

\newtheorem{theorem}{Theorem}[section]
\newtheorem{lemma}[theorem]{Lemma}
\newtheorem{corollary}[theorem]{Corollary}
\newtheorem{proposition}[theorem]{Proposition}

\newtheorem{remark}[theorem]{Remark}

\numberwithin{equation}{section}
 \newcommand{\C}{{\mathbb C}}
 
 \newcommand{\R}{{\mathbb {R}}}

 \newcommand{\Gn}{{\mathbf{G}}}
 \newcommand{\Godd}{{\Gn_{a,\lambda}}}
 \newcommand{\Gmixed}{{\Gn_{a,c_0}}}
 \newcommand{\X}{{\mathbf{X}}}
 \newcommand{\Xn}{{\mathbf{X}}}
 \newcommand{\Tn}{{\mathbf{T}}}

 \newcommand{\An}{{\mathbf{A}}}

 \newcommand{\A}{{\mathcal{A}}}
 
 \newcommand{\I}{{\mathcal{I}}}
 
 \newcommand{\sgn}{\operatorname{sgn}}
  \newcommand{\pv}{\operatorname{p.v}}

 \newcommand{\ps}{{\phi_2(x)}}
 \newcommand{\pt}{{\phi_3(x)}}
\begin{document}
%%%%%%%%%%%%%%%%%%%%%%%%%%%%%%%%%%%%%%%%%%%%%%%%%%%%%%%%%%%%%%%%%%%%%%%%%%%%%%%%%%%%%%%%%%
\title[Stability of self-similar solutions of 1D cubic Schr\"odinger equations]{
 On the stability of self-similar solutions of 1D cubic Schr\"odinger equations}
\author{S. Gutierrez$^1$ \and L. Vega$^2$}
\address{${\ }^{1}$S. Gutierrez, School of Mathematics, The Watson Building, University of Birmingham, Edgbaston,
Birmingham, B15 2TT, England.}
\email{S.Gutierrez@bham.ac.uk}
\address{${\ }^{2}$L. Vega, Departamento de Matem\'aticas, Universidad del Pa\'is Vasco,
Aptdo.~644, 48080 Bilbao, Spain.} \email{luis.vega@ehu.es}
\date{10th March 2011}
\begin{abstract}
In this paper we will study the stability properties of self-similar
solutions of  $1$-d cubic NLS equations with time-dependent
coefficients of the form
\begin{equation}
 \label{cubic}
 \displaystyle{
 iu_t+u_{xx}+\frac{u}{2} (|u|^2-\frac{A}{t})=0,
 \qquad A\in \R.
 }
\end{equation}
The study of the stability of these self-similar solutions is
related, through the Hasimoto transformation, to the stability of
some singular vortex dynamics in the setting of the {\it{Localized
Induction Equation}} (LIE), an equation modeling  the self-induced
motion of vortex filaments in ideal fluids and superfluids. We
follow the approach used by Banica and Vega that is based on the
so-called pseudo-conformal transformation, which reduces the problem
to the construction of modified wave operators for solutions of the
equation
$$
 iv_t+ v_{xx} +\frac{v}{2t}(|v|^2-A)=0.
$$
As a by-product of our results we prove that equation (\ref{cubic})
is well-posed in appropriate function spaces when the initial datum
is given by $u(0,x)= z_0 \pv \frac{1}{x}$ for some values of $z_0\in
\C\setminus \{ 0\}$, and $A$ is adequately chosen. This is in deep
contrast with the case when the initial datum is the Dirac-delta
distribution.
\end{abstract}
%%%%%%%%%%%%%%%%%%%%%%%%%%%%%%%%%%%%%%%%%%%%%%%%%%%%%%%%%%%%%%%%%%%%%%%%%%%%%%%%%%%%%%%%%%%%
\maketitle

%\thispagestyle{empty}
%-------------------------
\section{Introduction}
\label{introduction}
%-------------------------
In this paper, we study the stability properties of self-similar
solutions of the form
\begin{equation}
 \label{uf}
 u_{f}(t,x)= \frac{e^{\frac{ix^2}{4t}}}{\sqrt{t}} f\left( \frac{x}{\sqrt{t}}\right),
 \qquad x\in\R, \qquad t>0
\end{equation}
to the cubic nonlinear Schr\"odinger equations (NLS) in one
dimension:
\begin{equation}
 \label{eqs}
 iu_{t}+u_{xx}+\frac{u}{2} \left( |u|^2 -\frac{A}{t} \right)=0,
 \qquad A\in\R,
\end{equation}
that is solutions $u_f$ of the form (\ref{uf}) with $f$ a solution
of the equation
\begin{equation}
 \label{eqf}
 f''+i \frac{x}{2} f' + \frac{f}{2} (|f|^2-A)=0, \qquad A\in \mathbb{R}.
\end{equation}

Our main  motivation for the study of solutions of (\ref{eqs}) of
the form (\ref{uf}) comes from their connection to the singular
vortex  dynamics of what we refer to as ``self-similar" solutions to
the so-called {\it{Localized Induction Approximation}},  a geometric
flow in $\R^3$ modeling the dynamics of a vortex filament in ideal
fluids and superfluids.

The {\it {Localized Induction Approximation}}, often abbreviated LIA
or LIE, is des\-cribed by the following system of nonlinear
equations:
\begin{equation}
 \label{LIA}
 \Xn_t=\Xn_x\times \Xn_{xx},
\end{equation}
where $\Xn=\Xn(t,x)$ represents a curve in $\R^3$ with $t$ and $x$
denoting time and arclength, respectively. Using the Frenet equations we can also write
\begin{equation}
 \Xn_t=cb
\end{equation}
with $c$ and $b$ denoting the curvature and the binormal vector
respectively. For this reason the geometric PDE (\ref{LIA}) is also
referred to as {\it{binormal flow}}.

Equation (\ref{LIA}) was first proposed by Da Rios in 1906, and
rediscovered independently by Arms-Hamma and Betchov  in the early
$1960$s (see~\cite{DaR}, \cite{AH} and ~\cite{Be}), as an
approximation model for the self-induced motion of a vortex filament
in a $3$D-incompressible inviscid fluid. The use of the localized
induction equation to model the dynamical behaviour of a vortex in
superfluids such as ${\ }^{4}$He  started with the work by Schwarz
in 1985 (\cite{Sch}). In both, the classical and the superfluid
settings, the term localized induction approximation is used to
highlight the fact that this approximation only retains the local
effects of the Biot-Savart integral. We refer the reader to
\cite{B}, \cite{S}, \cite{AKO} and \cite{MB} for a detailed analysis
of the model and its limitations, and to the two papers by
T.~Lipniacki in \cite{Lip1} and \cite{Lip2} for further background
and references about the use of LIA in the setting of superfluid
helium.

Cubic NLS equations of the type (\ref{eqs}) are related to LIA
through the so-called Hasimoto transformation (see~\cite{Has}). This
connection is establised as follows: Let $\Xn=\Xn(t,x)$ be a regular
solution of LIA with associated curvature $c(t,x)$, and torsion
$\tau(t,x)$. Assuming that the curvature is strictly positive at all
points $x$, define the {\it{filament function}}
\begin{equation}
 \label{filamentfunction}
 u(t,x)= c(t,x)\text{exp}\, \left(
 i\int_{0}^{x} \tau(t, x')\, dx'
 \right).
\end{equation}
Then $u$ solves the nonlinear Schr\"odinger equation
\begin{equation}
 \label{sch}
 iu_t +u_{xx} +\frac{u}{2} (|u|^2-A(t))=0,
\end{equation}
where $A(t)$ is a time-dependent function which depends on the
values of $c(t,x)$ and $\tau(t,x)$ at $x=0$. Precisely,
\begin{equation}
 \label{At}
 A(t)= \left(
 2\frac{c_{xx}-c\tau^2}{c}+c^2
 \right)(t,0).
\end{equation}
Our analytical study of solutions of LIA started in \cite{GRV}, and
\cite{GV}, where the existence of solutions of LIA which develop a
singularity in finite time was established\footnote{See
also~\cite{Bu}, \cite{Lip1}, and \cite{Lip2}.}. The study of the
stability properties of the singular dynamics leading to the
formation of a corner singularity in finite time found in \cite{GRV}
was carried out by V.~Banica and L.~Vega in the  papers \cite{BV1},
\cite{BV2}, and \cite{BV3}.

Here, we are concerned with the singular dynamics found in
\cite{GV}. In particular, in \cite{GV}, solutions of LIA of the form
\begin{equation}
 \label{sol2}
  \Xn(t,x)= e^{\frac{\mathcal{A}}{2}\log t}\, \sqrt{t} \Gn(x/\sqrt{t}),
  \qquad t>0
\end{equation}
with $\mathcal{A}$ a real antisymmetric $3\times 3$ matrix of the
form
\begin{equation} \label{A}
 \mathcal{A}=
 \left(
 \begin{array}{ccc}
  0 & -a & 0\\
  a & 0 & 0\\
  0 & 0 & 0
 \end{array}
 \right),
 \qquad a\in\R
\end{equation}
are found to converge to a singular initial data $\Xn(0,x)$. The
precise statement of the result is the following:
\begin{proposition}
\label{spiral}
(See~\cite[Proposition 1 and 2]{GV}) For any
given $a\in\R$, and $\Gn$ solution of
\begin{equation}
 \label{P1}
  \Gn''=\frac{1}{2}(\mathcal{I}+\mathcal{A})\Gn\times \Gn'
\end{equation}
associated to an initial data $(\Gn(0), \Gn'(0))$ satisfying
\begin{equation}
 \label{P2}
  |\Gn(0)|=1
  \qquad {\hbox{and}}\qquad
  (\mathcal{I} +\mathcal{A}) \Gn(0)\cdot \Gn'(0)=0,
\end{equation}
define
\begin{equation}
 \label{P3}
  \Xn_a(t,x)= e^{\frac{\mathcal{A}}{2}\log t}\, \sqrt{t} \Gn(x/\sqrt{t}),
  \qquad t>0,
  \qquad {\hbox{with}}\qquad
  \mathcal{A}=
  \left(
  \begin{array}{ccc}
   0 & -a & 0\\
   a & 0 & 0\\
   0 & 0 & 0
  \end{array}
  \right).
\end{equation}
Then, $\Xn_a(t,x)$ is an analytic solution of LIA for all $t>0$, and
there exist non-zero vectors $\An^{+}$ and $\An^{-}\in \R^3$ such
that\footnote{$\chi_{E}(x)$ denotes the characteristic function of a
Lebesgue measurable set $E$.}
$$
 \lim_{t\rightarrow 0^{+}} \Xn_a(t,x)= xe^{\A\log|x|} (
 \An^{+}\chi_{[0,+\infty)}(x) + \An^{-}\chi_{(-\infty, 0]}(x)
 ):\,=\Xn_a(0,x)
$$
with
$$
 |\X_a(t,x)- xe^{\A\log |x|} \An^{\pm}|\leq
 2\sqrt{t}\left(  \sup_{x\in\R} |c(x)|  \right).
$$
Here, $c(x)$ is the curvature of the curve $\Gn(x)=\Xn(1,x)$, which is always bounded.
\end{proposition}
Solutions of the form (\ref{sol2}) have also been considered by T.~Lipniacki
(see~\cite{Lip1} and \cite{Lip2}) in the setting of the flow defined by
$$
 \Xn(t,x)=\beta \Xn_x\times \Xn_{xx}+\alpha \Xn_{xx}, \qquad \alpha\neq 0
$$
modeling the motion of a quantum vortex in superfluid helium.

Proposition~\ref{spiral} asserts that the evolution of the solution
$\Gn(x)$  of (\ref{P1})-(\ref{P2}) under the relation (\ref{P3})
leads to a solution of LIA which converges as $t\rightarrow 0^{+}$
to an initial curve $\X_a(0,x)$ given by
$$
 \Xn_a(0,x)= xe^{\A\log|x|} (
 \An^{+}\chi_{[0,+\infty)}(x) + \An^{-}\chi_{(-\infty, 0]}(x)).
$$
The initial curve $\Xn_a(0,x)$ is the sum of two 3d-logarithmic
spirals with a common origin. The rotation axis of these spirals is
the OZ-axis under the condition that the matrix $\mathcal{A}$ is of
the form (\ref{A}). In the case when the parameter $a\neq 0$, the
singularity of the initial curve $\Xn_{a}(0, x)$ comes from the
non-existence of the limit as $x\rightarrow 0$ of its tangent vector
$\Tn_{a}(0,x)$.

The properties of the ``self-similar" solutions of LIA given by
Proposition~\ref{spiral} rely, through the Hasimoto transformation,
on the properties of their associated filament function defined by
(\ref{filamentfunction}). This connection plays a fundamental role
in the study of the properties of these solutions (see~ \cite{GV}).
In particular, and following the philosophy in \cite{BV2}, a first
step to understand the stability properties of these solutions is to
study the stability of their related filament function in the
setting of the cubic Schr\"odinger equations (\ref{sch}). This will
be our main interest here.

In order to find the filament function associated to the
``self-similar" solutions of LIA given by Proposition~\ref{spiral},
first notice that it is straightforward to verify that the curvature
and torsion associated to solutions of LIA $\Xn_{a}(t,x)$ of the
form (\ref{P3}) are of the self-similar form\footnote{This is the
reason why we refer to solutions of LIA of the form (\ref{sol2}) as
``self-similar" solutions.}
$$
 c(t,x)=\frac{1}{\sqrt{t}}c(x/\sqrt{t})
 \qquad {\hbox{and}}\qquad
 \tau(t,x)=\frac{1}{\sqrt{t}} \tau(x/\sqrt{t}),
$$
so their filament function is given by
(see~(\ref{filamentfunction}))
\begin{equation}
 \label{filament-self}
  u(t,x)=\frac{1}{\sqrt{t}} u\left( \frac{x}{\sqrt{t}}   \right)
  \qquad {\hbox{and}}\qquad
  A(t)=\frac{A}{t},
\end{equation}
with $A=A(1)$ defined by (\ref{At}).
Since $\Xn_a(t,x)$ is a solution of LIA, through the Hasimoto
transform, we know that its filament function (\ref{filament-self})
solves the NLS
\begin{equation}
 \label{eqs'}
 iu_t+u_{xx}+\frac{u}{2} (|u|^2-\frac{A}{t})=0
 \qquad {\hbox{with}}\qquad
 A=A(1).
\end{equation}
Thus the function $u(s)$ in (\ref{filament-self}) is a solution of
the complex ODE
\begin{equation}
 \label{u-self}
   u''-\frac{i}{2}(u+xu')+\frac{u}{2} (|u|^2-A)=0.
\end{equation}
Notice that, by introducing a new variable $f$ defined by
\begin{equation}
 \label{u-f}
 u(x)=f(x)e^{i\frac{x^2}{4}},
\end{equation}
equation (\ref{u-self}) becomes
\begin{equation}
 \label{eqf'}
 f''+i\frac{s}{2} f'+\frac{f}{2}(|f|^2-A)=0.
\end{equation}
Previous lines show that the filament function associated to a
solution of LIA $\Xn_a(t,x)$ given by Proposition~\ref{spiral} is of
the form
$$
 u(t,x)=\frac{e^{i\frac{x^2}{4t}}}{\sqrt{t}} \, f\left( \frac{x}{\sqrt{t}}  \right),
$$
with $f$ a solution of the second order ODE (\ref{eqf'}), and solves
the nonlinear Schr\"odinger equation (\ref{eqs'}).

Furthermore, in \cite{GV} it was proved that the constant $A$ in the
above equation is given in terms of the initial conditions $(\Gn(0),
\Gn'(0))$ and the parameter $a$ by the identity
\begin{equation}
 \label{A1}
 A=aT_3(0)+\frac{|(\I +\A)\Gn(0)|^2}{4},
\end{equation}
and a solution $\Xn_a(t,x)$ and its associated function $f$ are
related through the following identities (see~\cite[pp. 2114]{GV})
\begin{equation}
 \label{key-identities}
 |f|^2(x)= -aT_3(x)+A,
 \qquad {\hbox{and}}\qquad
 |f'|^2(x)=\frac{1}{4} |\A \Tn \times \Tn|^2(x)= \frac{a^2}{4}(1-T^2_{3}),
 %\qquad {\hbox{and}}\qquad
 %\Im  (\bar f f')(x)= -\frac{1}{2} \A \Tn \cdot \Tn' (x),
\end{equation}
where $\Tn=(T_1, T_2, T_3)$ is the tangent unitary vector associated
to  $\Gn(x)=\Xn(1,x)$.

Among all the possible solutions of LIA $\X_a(t,x)$ of the form
(\ref{P3}) given by Proposition~\ref{spiral}, and in order to
motivate further our main result in this paper, it is important to
mention two special cases. In what follows, we will define what we
refer to as {\it{odd-solutions}} and {\it{mixed-symmetry solutions}}
of LIA.

The following cases come from the symmetry properties of the
equation
\begin{equation}
 \label{G}
  \Gn''=\frac{1}{2}(\I+\A)\Gn\times \Gn',
  \qquad \A=\left(
  \begin{array}{ccc}
  0 & -a & 0 \\
  a & 0 & 0 \\
  0 & 0 & 0
  \end{array}
  \right).
\end{equation}

\noindent {\it{Odd solutions:}} For fixed $a\in\R$ and $-1\leq
\lambda \leq 1$, let $\Godd$ the solution of (\ref{G}) with   the
initial condition
\begin{equation} \label{odd-con}
\Godd(0)=(0,0,0) \qquad {\hbox{and}}\qquad (\Godd)'(0)=(0, \sqrt{1-\lambda^2}, \lambda).
\end{equation}
Then,
\begin{equation}
 \label{odd}
 \Godd(x)=-\Godd(-x)
\end{equation}
(notice that if $\Gn(x)$ is a solution of (\ref{G})) with the
initial condition (\ref{odd-con}), then the function
$\tilde\Gn(x)=-\Gn(-x)$ is also a solution).

We refer the solutions of LIA of the form (\ref{P3}) with $\Godd(x)$
the solution of (\ref{G})-(\ref{odd-con}) as {\it{odd solutions}}.

In Figure~1 and Figure~2, we display the graphics of different
solutions  $\Godd$ of (\ref{G}) associated to an initial data of the
form (\ref{odd-con}). The right-handside pictures represent the
solution near the point $x=0$. The curvature of the curves $\Godd$
at the point $x=0$ is zero.

%--------- Examples: odd-solutions----------------------------------%
 \begin{figure}[h]
  \label{Odd1}
  \begin{center}
   \scalebox{0.4}{\includegraphics{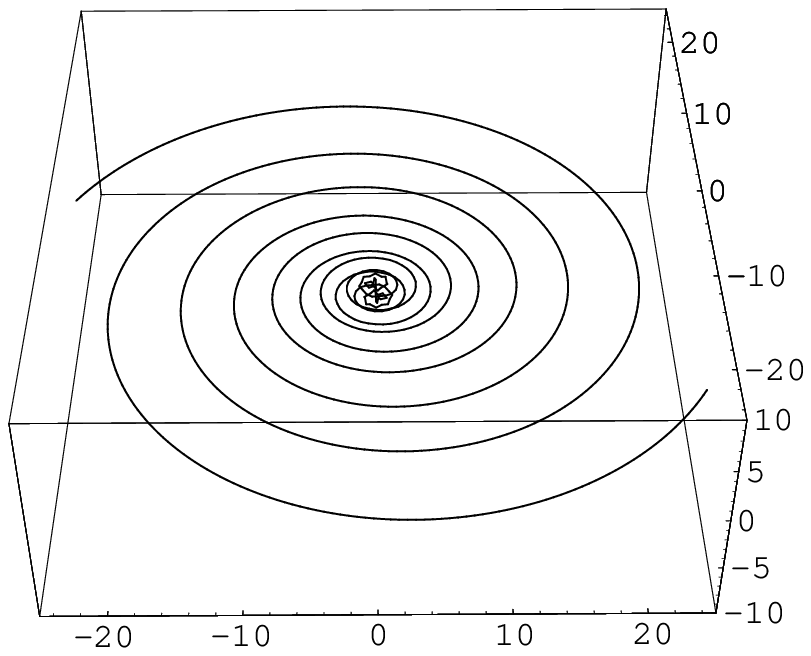}} %0.5
   \hspace{2cm}
   \scalebox{0.4}{\includegraphics{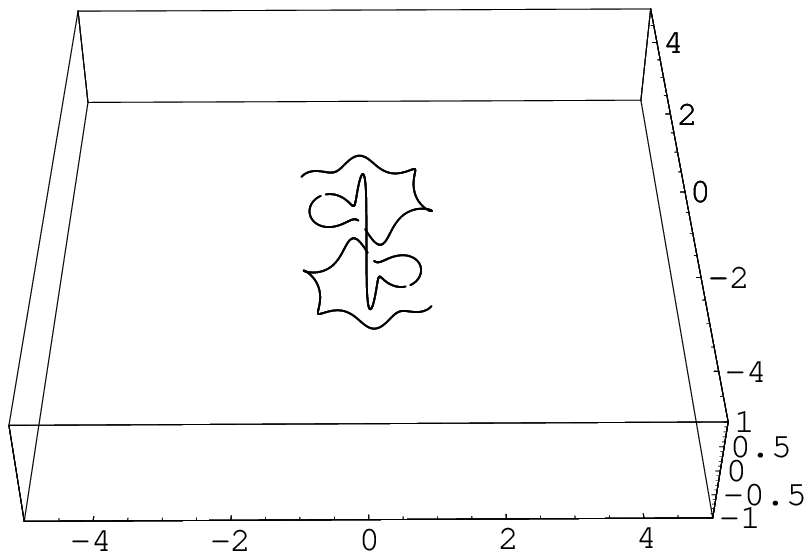}}  %0.5
  \end{center}
  \caption{Odd solutions. The vortex line
  $\Gn_{a,\lambda}$ corresponding to the solution of the system (\ref{G})-(\ref{odd-con}) with $a=10$ and $\lambda=0.956$.
}
 \end{figure}
% \vspace{-1.0cm}

 \begin{figure}[h]
  \label{Odd3}
  \begin{center}
   \scalebox{0.45}{\includegraphics{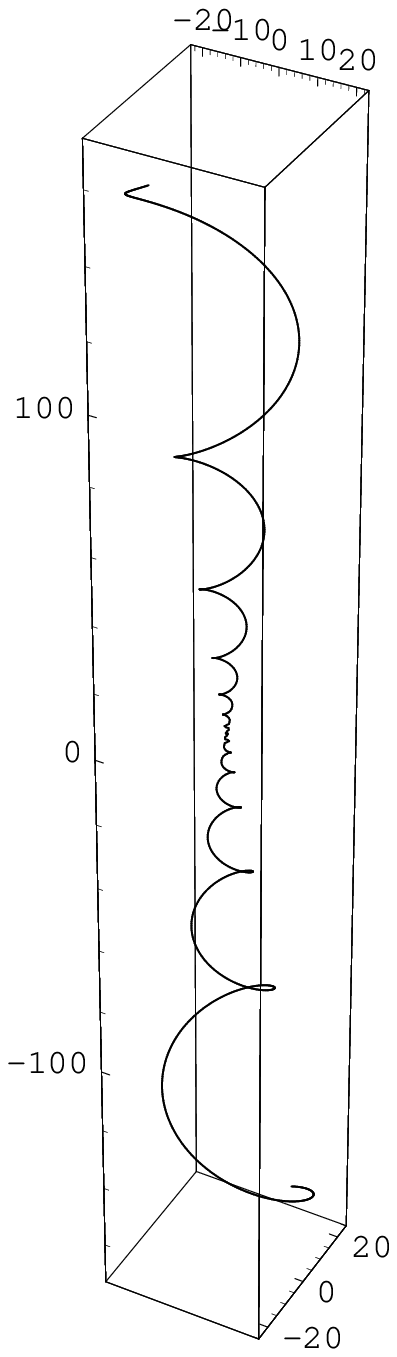}}  %0.55
   \hspace{3cm}
   \vspace{-.8cm}
   \scalebox{0.45}{\includegraphics{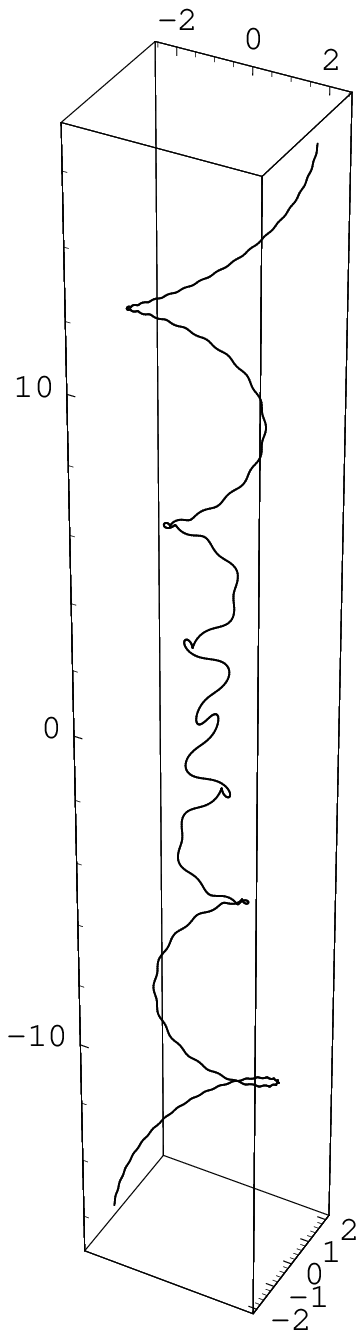}}  %0.55
  \end{center}
  \caption{Odd solutions. The vortex line $\Gn_{a,\lambda}$ corresponding to the solution of
  the system (\ref{G})-(\ref{odd-con}) with $a=10$ and $\lambda=-0.1$.
}
 \end{figure}
%------------------------------------------------------------------%
%

%
\noindent {\it{Mixed-symmetry solutions:}} For fixed $a\in\R$ and
$c_0>0$, let $\Gmixed$ the solution of (\ref{G}) with the initial
condition
\begin{equation} \label{mixed-con}
\Gmixed(0)=(\frac{2c_0}{\sqrt{1+a^2}},0,0) \qquad {\hbox{and}}\qquad (\Gmixed)'(0)=(0, 0, \pm 1).
\end{equation}
Then, $\Gmixed=(G_1, G_2,G_3)$ satisfies
\begin{equation}
 \label{mixed}
 \left\{
 \begin{array}{l}
  G_1(x)= G_1(-x)\\
  G_2(x)= G_2(-x)\\
  G_3(x)= -G_3(-x).\\
 \end{array}
 \right.
\end{equation}
This is a consequence of the fact that the equation (\ref{G}) and
the initial condition in (\ref{mixed-con}) remain unchanged by the
transformation $\Gn(x)=(G_1(x), G_2(x), G_3(x))\rightsquigarrow
(G_1(-x), G_2(-x), -G_3(-x))$.

We refer the solutions of LIA of the form (\ref{P3}) with
$\Gmixed(x)$ the solution of (\ref{G})-(\ref{mixed-con}) as
{\it{mixed-symmetry solutions}}.

Two examples of solutions of (\ref{G}) with initial data of the form
(\ref{mixed-con}) are plotted in Figures~3 and Figure~4. As before,
the r.h.s figure represents the curve $\Gn_{a, c_0}$ near the point
$x=0$.
%
%----------- Examples: Mixed symmetry-solutions-------------------------
%
%
 \begin{figure}[h]
 %\label{ }
  \begin{center}
   \scalebox{0.6}{\includegraphics{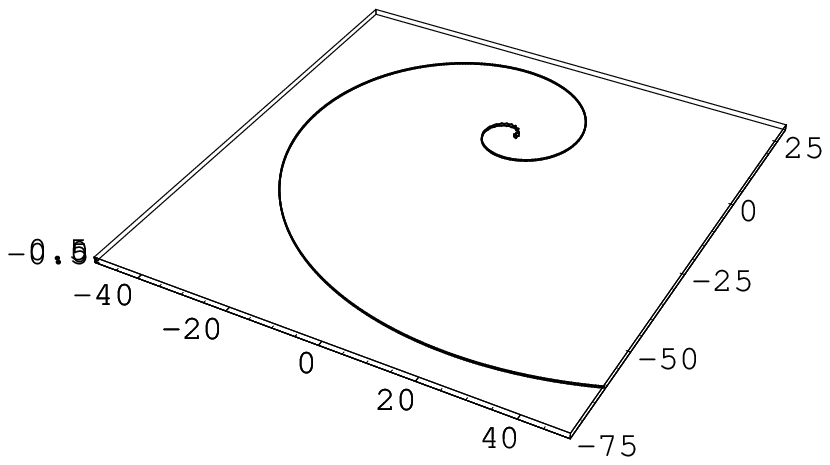}} %0.7
   \hspace{2cm}
   \scalebox{0.55}{\includegraphics{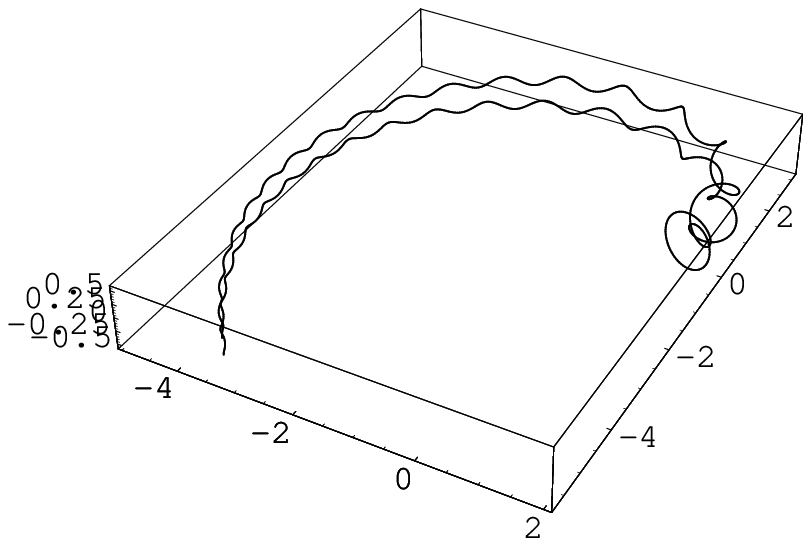}} %0.65
  \end{center}
   \caption{Mixed-symmetry solutions. The vortex line $\Gn_{a,c_0}$ corres\-ponding
to the solution of the system (\ref{G})-(\ref{mixed-con}) with $a=3$ and $c_0=1.8$.
}
 \end{figure}
 \begin{figure}[h]
%  \label{ }
  \begin{center}
   \scalebox{0.4}{\includegraphics{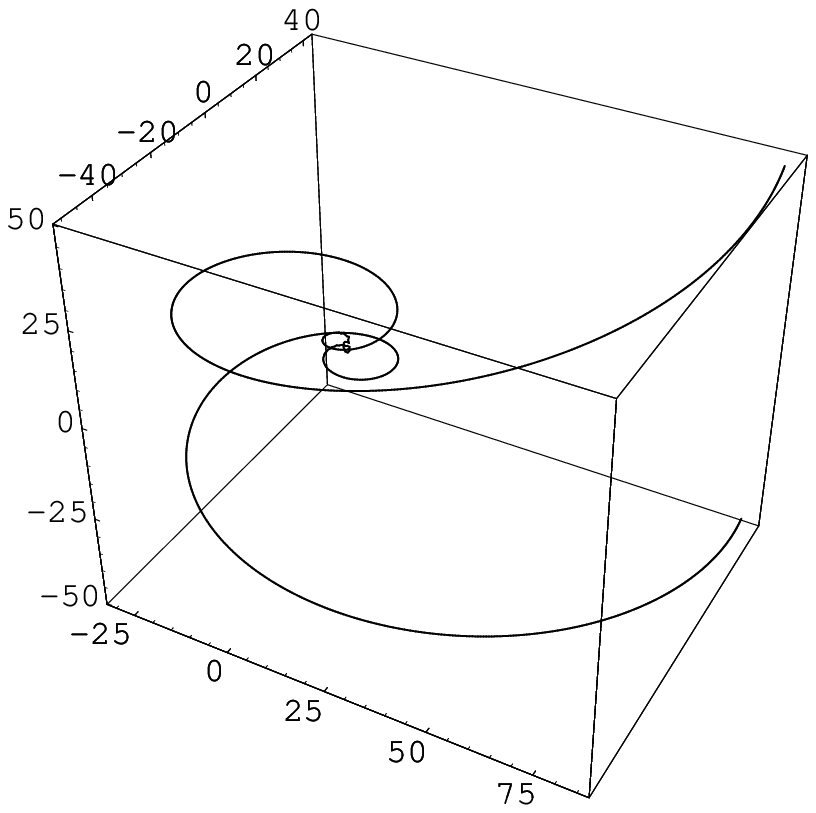}} %0.5
   \hspace{2cm}
   \scalebox{0.4}{\includegraphics{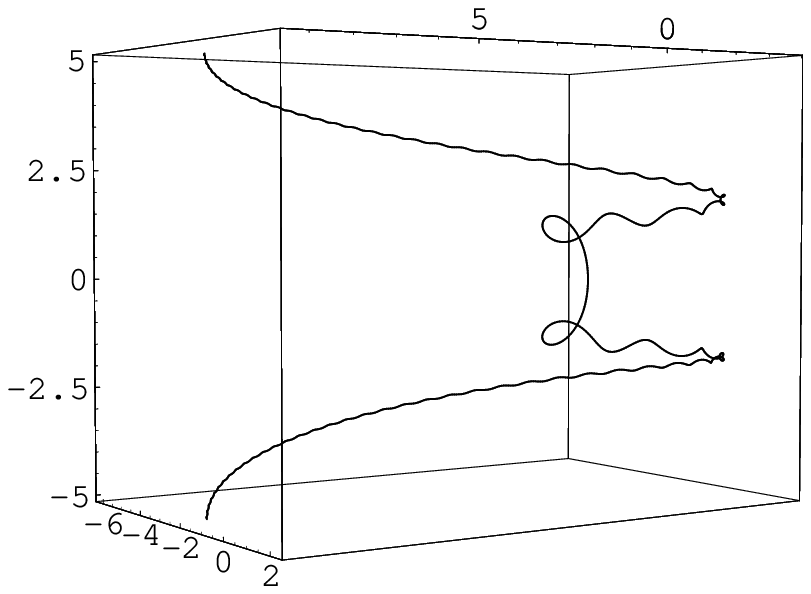}} %0.5
  \end{center}
  \caption{Mixed-symmetry solutions. The vortex line $\Gn_{a,c_0}$ corres\-ponding
to the solution of the system (\ref{G})-(\ref{mixed-con}) with $a=3$ and $c_0=0.4$.
%$\Gmixed(0)=(2c_0/\sqrt{1+a^2},0,0)$, $(\Gmixed)'(0)=(0,0,1)$, $a=3$, $c_0=0.4$.
}
 \end{figure}
%------------------------------------------------------------------------
%
%

\noindent Finally, observe that if $\Gn(x)=(G_1(x), G_2(x), G_3(x))$
is a solution of (\ref{G}), then the function $\tilde
\Gn(x)=(G_1(-x), -G_2(-x), G_3(-x))$ is a solution of
$$
 \Gn''= \frac{1}{2} (\I +\tilde \A)\Gn\times \Gn'
 \qquad {\hbox{with}}\qquad
 \tilde A=\left(
 \begin{array}{ccc}
  0 & a & 0\\
  -a& 0 & 0\\
  0& 0 & 0
 \end{array}
\right).
$$
As a consequence, in what follows we will assume w.l.o.g that $a\geq 0$.

Using the formulae (\ref{A1}) and (\ref{key-identities}), from the
initial conditions (\ref{odd-con}) it follows that the function $f$
associated to an odd solution of LIA is, through the Hasimoto
transformation (\ref{filamentfunction}) and the change of variables
(\ref{u-f}), a ({\it{odd}}) solution of
$$
 f''+i\frac{x}{2} f' +\frac{f}{2} (|f|^2-A)=0,
 %\qquad {\hbox{with}}
 \qquad A=a\lambda
$$
with initial conditions $(f(0), f'(0))$ satisfying
$$
 |f(0)|^2=0
 \qquad {\hbox{and}} \qquad
 |f'(0)|^2=\frac{a^2}{4}(1-\lambda^2),
 \qquad a>0\footnote{Notice that if $a=0$, then $|f(0)|=|f'(0)|=0$, so that $f\equiv 0$.}, \qquad -1\leq \lambda \leq 1.
$$
Analogously, from (\ref{A1}), (\ref{key-identities}),  and
(\ref{mixed-con}) it follows that the function $f$  associated to a
mixed-symmetry solution of LIA is a ({\it{even}}) solution of
$$
 f''+i\frac{x}{2} f' +\frac{f}{2} (|f|^2-A)=0,
 %\qquad {\hbox{with}}
 \qquad A=\pm a+ c^2_{0}
$$
with  initial conditions $(f(0), f'(0))$ satisfying
$$
 |f(0)|^2=c_0^2
 \qquad {\hbox{and}} \qquad
 |f'(0)|^2=0.
$$
From (\ref{filament-self}), (\ref{eqs'}), and the above argument it
follows that the filament function $u(t,x)$ associated to an odd
solution of LIA (respectively mixed-symmetry solution) is of the
form %(\ref{uf})
\begin{equation}
 \label{uf1}
 u_{f}(t,x)= \frac{e^{\frac{ix^2}{4t}}}{\sqrt{t}} f\left( \frac{x}{\sqrt{t}}\right),
 \qquad x\in\R, \qquad t>0
\end{equation}
\begin{equation}
 \label{f1}
 f''+i \frac{x}{2} f' + \frac{f}{2} (|f|^2-A)=0,
 %\qquad A=a\lambda
 %\qquad {\hbox{(resp. }} A=\pm a+c_0^2 {\hbox{)}},
\end{equation}
with $A=a\lambda$ (resp. $A=\pm a+c_0^2$), and solves the one
dimensional cubic Schr\"odinger equation
\begin{equation}
\label{sch1}
 iu_{t}+u_{xx}+\frac{u}{2} \left( |u|^2 -\frac{A}{t} \right)=0,
\end{equation}
with $A=a\lambda$ (resp. $A=\pm a+c_0^2$).

As we have already mentioned, this paper is devoted to the study the
stability properties of certain self-similar solutions $u_{f}$ in
(\ref{uf1}) of the 1D cubic Shr\"odinger equation (\ref{sch1}).

In order to give a precise statement of our results, we consider the
so-called pseudo-conformal transformation of (\ref{sch1}). Briefly,
given any solution $u$ of (\ref{sch1}), we define a new unknown $v$
as follows
\begin{equation}
 \label{v}
 u(t,x)= \mathcal{T}v(t,x)
 =\frac{e^{i\frac{x^2}{4t}}}{\sqrt{t}}
 \bar v \left( \frac{1}{t}, \frac{x}{t}   \right).
\end{equation}
Here, and elsewhere, an overbar denotes complex conjugation. Then
$v$ has to be a solution of
\begin{equation}
 \label{eqv}
 iv_{t} + v_{xx}+ \frac{v}{2t} \, (|v|^2-A)=0.
\end{equation}
In particular, solutions $u_f$ of (\ref{sch1}) correspond to
solutions $v_f$ of (\ref{eqv}) of the form
\begin{equation}
 \label{vf}
 v_{f}(t,x)= \bar f \left( \frac{x}{\sqrt{t}}  \right).
\end{equation}
Thus, we are reduced to prove the existence of appropriate
perturbations (modified wave operator) around the solutions $v_f$,
\medskip

The study of the stability properties of  solutions  of (\ref{eqv})
of the form (\ref{vf}) (and, consequently, of solutions of
(\ref{sch1}) of the form (\ref{uf1})) started in \cite{BV1}, and
carried on in  \cite{BV2} and \cite{BV3}. Precisely, in \cite{BV2},
the authors studied the stability of the  solution of (\ref{eqv})
with $A=c_0^2$ given by
$$
 v_{c_0}(t,x)= c_0.
$$
In \cite[Theorem~1.2]{BV2}), under the smallness assumption of the
parameter $c_0>0$, the authors prove that for any $t_0>0$, and any
given asymptotic state $u_{+}$ small in $L^1\cap L^2$ (w.r.t $t_0$
and $c_0$) the equation (\ref{eqv}) has a unique solution $v(t,x)$
in the interval $[t_0,\infty)$ which behaves like
$$
 v_1(t,x)= c_0+ e^{i\frac{c_0^2}{2}\log t} \left( e^{it\partial_{x}^{2}} u_{+}\right)(x)
$$
as t goes to infinity, in the sense that
\begin{equation}
 \label{lim-v_{c_0}}
 {\|v(t)- v_{1}(t)\|}_{L^2}=\mathcal{O}(t^{-\frac{1}{4}}),
 \qquad {\hbox{as}}\qquad
 t\longrightarrow \infty.
\end{equation}
Here $e^{it\partial_{x}^{2}}$ denotes the free propagator (see
notation below). In other words, they construct the so-called
(modified) wave operators. In \cite{BV3} this result is extended to first remove the smallness assumption on $c_0$ and
moreover to
consider also the asymptotic completeness of the scattering
operators. One of the fundamental ingredients in this paper is the
study of the linearized equation of (\ref{eqv}) around the constant
solution $v_{c_0}(t,x)=c_0$ ($A={c_0}^2$) given by
\begin{equation}
 \label{lin-c0}
 iz_t+z_{xx}+ \frac{c_0^2}{2t}(z+\bar z)=0.
\end{equation}
Notice that the coefficients in the above equation only depend on
$t$, and as a consequence this linearized equation can be analyzed
by computing the Fourier transform in space. Unfortunately, in our
case the linearized equation of (\ref{eqv}) around solutions $v_f$
of the form (\ref{vf}) is given by (see~(\ref{eqw}) below)
\begin{equation}
 \label{lin-eq}
 iz_t+z_{xx}+\frac{1}{2t}[
 (2|v_f|^2-A)z+ v_f^2 \bar z]=0,
\end{equation}
with coefficients that are also space
dependent\footnote{See~Proposition~\ref{f} below for the properties
of $v_f(t,x)=\bar f(x/\sqrt{t})$.}. This makes the analysis of the
linearized equation (\ref{lin-eq}) to be much more delicate.
Therefore, we put ourselves in the most simple situation. Firstly,
and as in \cite{BV2}, we will just consider the construction of the
wave operators. Secondly, we reduce our analysis to those
self-similar solutions $v_f(t,x)=\bar f(x/\sqrt{t})$  that have the
extra property that\footnote{The existence of $|f|_{\pm\infty}$ was
established in~\cite{GV}, see Proposition~\ref{f} below.}
$$
   |f|_{+\infty}=|f|_{-\infty} \qquad({\hbox{that is}} \quad |f|(+\infty)=|f|(-\infty)),
$$
and in particular those that the function $f$ is an odd or even
function. Even under this assumption the equation (\ref{lin-eq}) is
not so easy to handle. In fact, as we will see in the statement of
the main theorem below, we can not consider the asymptotic state
$u_{+}$ to be in $L^1\cap L^2$ as in \cite{BV2} and some weighted
$L^2$-spaces are necessary. This implies some loss in the rate of
decay given in (\ref{lim-v_{c_0}}). The main difficulty comes from
the appearance in the Duhamel term (\ref{Duhamel}) of $v_{f}^2$,
which depends on both the spatial and time variable. This differs
from the situation in \cite{BV2} where $v_f^2(t,x)= c_0^2$.

Before stating our results, we introduce some
conventions and function spaces.
We denote by $L^{2}(|x|^\gamma)$ and $L^2(\langle x \rangle^\gamma)$
the $L^2$-spaces with Lebesgue measure replaced by $|x|^{\gamma}\,
dx$, and $\langle x \rangle^\gamma\, dx=(1+|x|^2)^{\gamma/2}\, dx$,
respectively, i.e.,
$$
 L^{2}(|x|^\gamma)=\{\phi:\R\longrightarrow \C
 \ : \ {\|\phi\|}_{L^{2}(|x|^\gamma)}= \left(\int_{\R}  |\phi(x)|^2 |x|^\gamma\, dx\right)^{1/2}<\infty  \},
$$
and
$$
  L^{2}(\langle x \rangle^\gamma)=\{\phi:\R\longrightarrow \C
 \ : \ {\|\phi\|}_{L^{2}(\langle x \rangle^\gamma)}= \left(\int_{\R}  |\phi(x)|^2 (1+|x|^2)^{\gamma/2}\, dx\right)^{1/2}<\infty  \}.
$$
For $s\in \mathbb{N}^{\star}$, the Sobolev space $H^s$ is defined by
$$
  H^s=\{ f\in \mathcal{S}(\R)\ : \ \nabla^k f\in L^2(\R), \ \forall\, 0\leq k\leq s\}.
$$
The Fourier transform of $v$, $\hat v$, %or $\mathcal{F}v$
is defined by
$$
 \hat v(\xi)=\frac{1}{2 \pi}\int_{\R} e^{-ix\cdot \xi} v(x)\, dx,\qquad
$$
$e^{it\partial^2_{x}}u_0$ denotes the solution to the initial value
problem for the free 1D Schr\"odinger equation with initial data
$u_0$, defined by
\begin{equation}
 \label{sh-solution}
 \left(  e^{it\partial_{x}^{2}}  u_0\right)(x):= \int_{\R} e^{ix\xi}e^{-i\xi^2 t} \widehat{u_0}(\xi)\, d\xi,
\end{equation}
or, equivalently,
\begin{equation}
 \label{sh-solution1}
 \left(  e^{it\partial_{x}^{2}}  u_0\right)(x):=
 \frac{1}{\sqrt{4\pi i t}} \int_{\R} u_0(y) e^{i\frac{(x-y)^2}{4t}}\, dy.
\end{equation}
For any $u_{+}$, and $f$ solution of (\ref{eqf}) such that
$|f|_{+\infty}=|f|_{-\infty}$, we define $\tilde v_{f}$ by
\begin{equation} \label{tildevf}
 \tilde v_{f}(t,x)= v_{f}(t,x)+ e^{i\frac{\alpha}{2} \log t}
 \left( e^{it\partial_{x}^2} u_{+}\right) (x),
\end{equation}
with
$$
 v_{f}(t,x)=\bar f\left(  \frac{x}{\sqrt{t}} \right)
 \quad {\hbox{and}}\quad
 \alpha=2|f|^2_{\infty}-A.\footnote{
 If $|f|_{+\infty}=|f|_{-\infty}$ ($|f'|_{+\infty}=|f'|_{-\infty}$), then we will denote $|f|_{\pm\infty}$ by $|f|_{\infty}$ (respectively, $|f'|_{\pm\infty}$ by $|f'|_{\infty}$).
 %The existence of $|f|_{\pm\infty}$ and $|f'|_{\pm\infty}$ is established in \cite{GV}, see Proposition \ref{f}  below.
 }
$$

Our main result is the following
%%%%%%%%%%%%%%%%%%%%%%%%%%%%%%%%%%%%%%%%%%%%%%%%%%%%%%%%%%%%%%%%%%%%%%%%%%%%%%%%%%%%%%%%%%%%
\begin{theorem}
 \label{T1}
 Let $t_0>0$, and $0<\gamma<1$. There exists a (small) positive
 constant $B_0$, such that for any  $A$ and any $f$ solution of
 $$
  f''+i\frac{x}{2} f'+\frac{f}{2}(|f|^2-A)=0
 $$
such that $|f|_{-\infty}=|f|_{+\infty}$ with
${\|f\|}_{L^{\infty}}\leq B_0$, and $u_{+}$ small in $L^1\cap
L^2(\langle x \rangle^\gamma)$ with respect to $B_0$, $t_0$, and
$f$, the equation
\begin{equation}
 \label{T1a}
  iv_t+v_{xx}+\frac{v}{2t}(|v|^2-A)=0
\end{equation}
has a unique solution $v(t,x)$ in the time interval $[t_0, \infty)$
such that
$$
v-\tilde v_f\in
\mathcal{C}\left(
 [t_0,\infty), L^2(\mathbb{R})\right)
 \cap L^4\left(  [t_0,\infty), L^\infty(\mathbb{R})  \right).
$$
Moreover, the solution $v$  satisfies
\begin{equation}
 \label{T1b}
 {\| v-\tilde v_f \|}_{L^2(\mathbb{R})}+
 {\| v-\tilde v_f \|}_{L^4((t,\infty), L^\infty(\mathbb{R}))}=
 \mathcal{O}\left(  \frac{1}{t^{\frac{\gamma}{4}}} \right),
\end{equation}
as $t$ goes to infinity.

In addition, if $\partial_x u_{+}\in L^1\cap L^2(\langle x
\rangle^{\gamma})$, with $0<\gamma<1$, then $v-\tilde v_f \in H^1$
and
\begin{equation}
 \label{T1c}
{\| v-\tilde v_f  \|}_{H^1}= \mathcal{O}\left(  \frac{1}{t^{\frac{\gamma}{4}}} \right), \qquad t\rightarrow \infty.
\end{equation}
\end{theorem}
%%%%%%%%%%%%%%%%%%%%%%%%%%%%%%%%%%%%%%%%%%%%%%%%%%%%%%%%%%%%%%%%%%%%%%%%%%%%%%%%%%%%%%%%%%%%%%%%%%%%%%%%%%%%%%
%
The above result asserts  the existence of the modified wave
operator in the time interval $[t_0,\infty)$ with $t_0>0$, for any
given final data $u_{+}$ in $L^1\cap L^2(\langle x
\rangle^{\gamma})$ with $0<\gamma<1$, and any $f$ solution of
(\ref{eqf}) such that $|f|_{+\infty}=|f|_{-\infty}$, under smallness
conditions on ${\| f \|}_{L^\infty}$ and the data $u_{+}$.

%%%%%%%%%%%%%%%%%%%%%%%%%%%%%%%%%%%%%%%%%%%%%%%%%%%%%%%%%%%%%%%%%%%%
\begin{remark}
 \label{delta=0}
 As we said before, the new difficulties in the proof of this result with respect
  to those in \cite{BV2} come from the space dependence of the coefficients of linearized equation (\ref{lin-eq}). There is a particular case
  where this equation is as simple as (\ref{lin-c0}). This happens when the phase function $\phi_2(x)=(|f|_{\pm}^2-A)\log|x|$
  in Proposition~\ref{f} is identically zero, that is when $|f|^2_{\pm\infty}=A$.

  It turns out that in this particular case the corresponding curve $\Xn(t,x)$ is asymptotically flat at infinity, that
  is $T_3(\pm\infty)=0$ with $T_3(x)$ being the third component of the tangent vector to the curve $\Xn(t,x)$ (see~(\ref{key-identities})).
  In this situation, one could expect that the stronger results proved in \cite{BV3}
  could also be extended to this case. This will be studied elsewhere.
\end{remark}
%%%%%%%%%%%%%%%%%%%%%%%%%%%%%%%%%%%%%%%%%%%%%%%%%%%%%%%%%%%%%%%%%%%%%

Once $v$ has been constructed, we recover $u$ through the
pseudo-conformal transformation (\ref{v}). Precisely,  defining
$\tilde u_f$ as
\begin{equation}\label{u1}
 \tilde u_f (t,x)= \frac{e^{i\frac{x^2}{4t}}}{\sqrt{t}} f\left(  \frac{x}{\sqrt{t}}\right)
 +
 {\sqrt{\pi i}} \, e^{i\frac{\alpha}{2}\log t} \,\widehat{\overline{u_{+}}}\left( -\frac{x}{2}\right),
 \qquad \alpha=2|f|^2_{\infty}-A,
\end{equation}
as a consequence of Theorem~\ref{T1} we obtain the following:
%%%%%%%%%%%%%%%%%%%%%%%%%%%%%%%%%%%%%%%%%%%%%%%%%%%%%%%%%%%%%%%%%%%%%%%%%%%%%%%%%%%%%%%%%%%%%%
\begin{theorem}
 \label{T2}
 Let $\tilde t_0>0$, and  $0<\gamma<1$. There exist a (small) positive constant $B_0$, such that for
 any $A$ and any $f$  solution of
 $$
  f''+i\frac{x}{2} f'+\frac{f}{2}(|f|^2-A)=0
 $$
such that $|f|_{+\infty}=|f|_{-\infty}$ with
${\|f\|}_{L^{\infty}}\leq B_0$, and $u_{+}$ small in $L^1\cap
L^2(\langle x\rangle^\gamma)$ with respect to $B_0$, $\tilde t_0$
and $f$, the equation
\begin{equation}
 \label{eqs1}
  iu_t+u_{xx}+\frac{u}{2}(|u|^2-\frac{A}{t})=0.
\end{equation}
has a unique solution $u(t,x)$ in the interval $(0,\tilde t_0]$ such
that
$$
 u-\tilde u_f\in \mathcal{C}\left( (0, \tilde t_0], L^2(\R)\right) \cap
 L^4\left( (0, \tilde t_0], L^\infty(\R)  \right).
$$
Moreover, as $t$ goes to zero, the solution $u$  satisfies,
\begin{equation}
 \label{C1a}
 {\| u-\tilde u_f \|}_{L^2(\mathbb{R})}+
 {\| u-\tilde u_f \|}_{L^4((0,t), L^\infty(\mathbb{R}))}=
 \mathcal{O}( t^{\frac{\gamma}{4}}).
\end{equation}
In particular, as $t$ goes to zero
\begin{equation}
 \label{C1b}
 {\left\|
  u(t,\cdot) - \frac{e^{i\frac{(\cdot)^2}{4t}}}{\sqrt{t}}\, f\left(  \frac{\cdot}{\sqrt{t}} \right)
 \right\|}_{L^2(\R)}=\mathcal{O}(1),
 \qquad {\hbox{and}}
\end{equation}
\begin{equation}
 \label{C1c}
 \left\|
  \left| u(t,\cdot) - \frac{e^{i\frac{(\cdot)^2}{4t}}}{\sqrt{t}}\, f\left(  \frac{\cdot}{\sqrt{t}} \right)
 \right|^2  -
 \left|
  \sqrt{\pi}\,  \widehat{{\overline {u_+}}}\left( - \frac{\cdot}{2} \right)   \right|^2
 \right\|_{L^1(\R)}=\mathcal{O}(t^{\frac{\gamma}{4}}),
\end{equation}
but the limit of $\displaystyle{
u(t,x)-\frac{e^{i\frac{x^2}{4t}}}{\sqrt{t}}\, f\left(
\frac{x}{\sqrt{t}} \right) }$ does not exist in $L^2(\R)$ as $t$
goes to zero unless $\alpha=  2|f|^2_{\infty}-A  =0$.

Finally, if in addition $\partial_{x} u_{+} \in L^1\cap L^2(\langle
x \rangle^{\gamma})$, then

\begin{equation}
 \label{C1d1}
 |u(t,x)|\leq
 \frac{2}{\sqrt{t}} \left|
 f\left( \frac{x}{\sqrt{t}}\right)
 \right|,
\end{equation}
for all $x\in \R$ and $0<t<1$, and if $x\neq 0$ there exists
$t^{*}(x)>0$ such that for $0<t<t^{*}(x)$
\begin{equation}
 \label{C1d2}
 \frac{1}{2\sqrt{t}}\, \left|  f\left(  \frac{x}{\sqrt{t}} \right) \right|\leq
 |u(t,x)|.
\end{equation}
\end{theorem}
%%%%%%%%%%%%%%%%%%%%%%%%%%%%%%%%%%%%%%%%%%%%%%%%%%%%%%%%%%%%%%%%%%
%
\begin{remark}
 \label{alpha=0}
 The case when the solution $f$ satisfies the condition $2|f|^2_{\infty}-A=0$ (that is, when $\alpha=0$)
 deserves a special attention. First of all, from (\ref{u1}) and (\ref{C1a}), we observe that $\tilde u_f(t,\cdot)$,
 and then $u(t,\cdot)$, will have a limit as long as such a limit exists for
 $$
  u_f(t,x)=\frac{e^{i\frac{x^2}{4t}}}{\sqrt{t}} f\left( \frac{x}{\sqrt{t}}  \right).
 $$
We will see in Section~\ref{section-pv} that precisely under the
same condition $2|f|^2_{\infty}-A=0$,  $u_f(t,\cdot)$ converges in
the distribution sense to $z_0\, \pv \frac{1}{x}$. As a consequence,
the initial value problem (IVP for short) given by (\ref{eqs1}) and
$$
 u(0,x)= z_0\, \pv \frac{1}{x} + \sqrt{\pi i}\widehat{\overline{u_{+}}}\left( -\frac{x}{2}\right)
$$
is well-posed in appropriate function spaces. See Theorem~\ref{T4}
in Section~\ref{section-pv} for the precise statement.
\end{remark}
%
%
%%%%%%%%%%%%%%%%%%%%%%%%%%%%%%%%%%%%%%%%%%%%%%%%%%%%%%%%%%%%%%%%%%%%
%
\begin{remark}
 \label{smallness}
 Recall that, for the 1d-cubic NLS equations associated to solutions of LIA of
 the form (\ref{sol2})-(\ref{A}), the coefficient $A$ is linked to
 the initial conditions and the parameter $a$ through the identity
 (\ref{A1}).

From the identity (\ref{A1}) and the conservation law for $f$ stated
in Proposition~\ref{f}, we conclude that the smallness assumption
for ${\| f \|}_{L^\infty}$ can be achieved  by considering initial
data $(\Gn(0), \Gn'(0), a)$ sufficiently small.
\end{remark}
%%%%%%%%%%%%%%%%%%%%%%%%%%%%%%%%%%%%%%%%%%%%%%%%%%%%%%%%%%%%%%%%%%%%
Notice that the solutions $u$ given by Theorem~\ref{T2} do not have
a trace at $t=0$ (see comment after (\ref{C1c})). Nevertheless,
associated to these solutions we are able to construct a family of
curves $\Xn(t,x)$ solutions of LIA which do have a limit at $t=0$.
The precise statement of the result is the following:
\begin{corollary}
 \label{T3}
 % Existence of the trace $\Xn_0(x)$ for the perturbed solutions.
 Let $0<\gamma<1$ and  $u_{+},\ \partial_x u_{+}\in  L^1\cap L^2(\langle x
 \rangle^{\gamma})$. Then, under the smallness assumptions of
 Theorem~\ref{T2}, for $0<t<\tilde t_0$ there exists a unique solution
 $\Xn(t,x)$ of LIA such that the filament function of $\Xn(t,x)$ is
 the function $u(t,x)$ given by Theorem~\ref{T2}, $\Xn(\tilde t_0, 0)=(0, 0, 0)$
 and $ \Xn_{x}(\tilde t_0,0)=(1,0,0)$.

 Moreover,
 \begin{itemize}
 \item[{\it{i)}}] the curvature of the curve $\Xn(t,x)$, $c(t,x)$, satisfies
 $$
  |c(t,x)|\leq
  \frac{c_1}{\sqrt{t}}
 $$
 for all $x\in \R$ and $0<t<1$, and if $x\neq 0$, there exists
 $t^{*}(x)>0$ such that for all $0<t< t^{*}(x)$
 $$
    \frac{c_2}{\sqrt{t}} \leq |c(t,x)|.
 $$
 \item[{\it{ii)}}] In addition, there exists  a unique $\Xn_0(x)$ such that
 $$
  |\Xn(t,x)-\Xn_0(x)|\leq c_3\sqrt{t},
 $$
 uniformly on the interval $(-\infty, \infty)$, with $\Xn_0(x)$
 a Lipschitz continuous function.
 \end{itemize}
 Here, $c_1, c_2$, and $c_3$ are non-negative  constants.
\end{corollary}
The proofs of all these results are given in Section~2. In Section~3
we state and prove Theorem~\ref{T4} about the well-posedness of the
IVP given by (\ref{eqs}) and $u(0,x)= z_0 \pv \frac{1}{x}$ for some
values of $z_0\in \C\setminus \{0\}$. The question of well-posedness
of the $1d$ cubic NLS for spaces that include $L^2$ was started in
\cite{VV}, and then extended in \cite{Gru} to all the range of
subcritical scales. In fact it was proved in \cite{KPV} that when
the initial datum is given by Dirac-delta function, the IVP is
ill-posed due to the appearance of a logarithmic correction in the
phase. This phase can be canceled out by modifying the equation with
an extra factor $A(t)=c_0/t$ for some constant $c_0$ as in equation
(\ref{eqs}). As we said before this modification naturally appears
when the $1d$ cubic NLS is obtained from LIA through the Hasimoto
transformation. However, even with this modification it was proved
in \cite{BV2} and \cite{BV3} that the problem is still ill-posed for
the Dirac-delta. The reason is the same that the one pointed out in
the statement of Theorem~\ref{T2}.  In \cite{KPV} the ill-posedness
was obtained as a consequence of the invariance of NLS by  galilean
transformations. The same ideas imply that, if the notion of
well-posedness includes the uniform continuity of the map
datum-solution, then the class of  Sobolev spaces of negative index
has to be excluded, see also \cite{KPV}. However the existence of a
priori upper bounds for the Sobolev norm of the solution, in terms
of the Sobolev norm of the datum, for arbitrarily large data, and
for sufficiently short time can be proved, see \cite{CCT}  and
\cite{KT}.

%%%%%%%%%%%%%%%%%%%%%%%%%%%%%%%%%%%%%%%%%%%%%%%%%%%%%%%%%%%%%%%%%%%%%%%%%%%%%%%%
\section{Modified wave operators in mixed norm spaces.}
\label{construction}
%%%%%%%%%%%%%%%%%%%%%%%%%%%%%%%%%%%%%%%%%%%%%%%%%%%%%%%%%%%%%%%%%%%%%%%%%%%%%%%%
%
In order to find the ``appropriate" modified  wave operators for
$v-v_{f}$ in the setting of the equation (\ref{eqv}), we follow the
arguments given in \cite{BV2}. Briefly,  write
$$
 v=v_{f}+w,
$$
with $v$ and $v_{f}$ solutions of the Schr\"odinger equation
$$
 iv_{t}+ v_{xx}+ \frac{v}{2t} \, \left( |v|^2 -A  \right)=0.
$$
Then, $w$ has to be a solution of
$$
 iw_{t}+ w_{xx}+\frac{1}{2t}
 \left[
 (|v_f|^2-A)w+(v_{f}\bar w + \bar v_{f} w +|w|^2) (v_{f}+w)
 \right]=0,
$$
or,
\begin{equation}
 \label{eqw}
  iw_{t}+ w_{xx}+\frac{1}{2t}
 \left[
 (2|v_f|^2-A)w+v_{f}^2\bar w + 2 v_{f} |w|^2 + \bar v_{f} w^2+|w|^2 w
 \right]=0.
\end{equation}
The linear term $(2|v_f|^2-A) w/2t$ (in the above equation) is
resonant and, as we will continue to show, it is the responsible for
a logarithmic correction of the phase.

In order to deal with the resonant structure of this term, here and
in what follows, we assume that $f$ is such that
$|f|_{+\infty}=|f|_{-\infty}$, and we write the above equation
equivalently as
\begin{multline}
  iw_{t}+ w_{xx}+\frac{1}{2t}
 \left[
 (2|f|^2_{\infty}-A)w+ 2(|v_f|^2-|f|^2_{\infty}) w+v_{f}^2\bar w +
 \right.
 \nonumber \\
 \left.
 2 v_{f} |w|^2 + \bar v_{f} w^2+|w|^2 w
 \right]=0.
\end{multline}
Observe that $|f|^2_{\pm\infty}$ is nothing but the limit of
$|v_f(t,x)|^2$ as $x\rightarrow \pm\infty$, i.e.
$$
 |f|^2_{\infty}=\lim_{x\rightarrow \pm \infty}
 \left|\bar f\left( \frac{x}{\sqrt{t}}\right)\right|^2=
 \lim_{x\rightarrow\pm\infty} |v_f(t,x)|^2.
$$
Then, if we define a new function $u$ as
$$
 u(t,x)=w(t,x)e^{-i\frac{\alpha}{2}\log t}, \quad t>0,\quad {\hbox{with}}\quad
 \alpha=2|f|^2_{\infty}-A,
$$
the function $u$ has to be a solution of
\begin{multline}
   iu_{t}+ u_{xx}+
 \frac{1}{t} (|v_f|^2-|f|^2_{\infty})u +
 \frac{v_f^2}{2t}e^{-i\alpha \log t} \bar u +
  \nonumber \\
 \frac{1}{2t}
 \left[
 2 v_{f} e^{-i\frac{\alpha}{2}\log t}|u|^2 + \bar v_{f} e^{i\frac{\alpha}{2} \log t}u^2+|u|^2 u
 \right]=0.
\end{multline}
Although the first linear term in (\ref{equ}), that is
$$
  (|v_f|^2-|f|^2_{\infty})\frac{u}{t}
$$
is still resonant, the structure of $|v_{f}(t,x)|^2-|f|^2_{\infty}$
allows us to treat this term as a perturbative  term in the Duhamel
formula for the solution and, to consider as initial guess
$$
 u(t,x)\thickapprox \left( e^{it\partial^2_{x}} u_{+}\right) (x),
 \qquad {\hbox{as}}\qquad
 t\rightarrow +\infty.
$$
Summing up, for any given asymptotic state $u_{+}$, we consider the
following guess for the perturbation
$$
 e^{i\frac{\alpha}{2}\log t} \left( e^{it\partial^2_{x}} u_{+} \right)(x),
 \qquad \alpha=2|f|^2_{\infty}-A,
$$
and define $\tilde v_f$ to be
$$
 \tilde v_{f}(t,x)= v_{f}(t,x)+ e^{i\frac{\alpha}{2} \log t} \left( e^{it\partial_{x}^2} u_{+}\right)(x),
 \quad {\hbox{with}}\quad
 \alpha=2|f|^2_{\infty}-A.
$$
%
%%%%%%%%%%%%%%%%%%%%%%%%%%%%%%%%%%%%%%%%%%%%%%
\subsection{Preliminaries.}
\label{Preliminaries}
%%%%%%%%%%%%%%%%%%%%%%%%%%%%%%%%%%%%%%%%%%%%%%%
%

Using the notation introduced previously, given $u_{+}$ and $f$
solution of (\ref{eqf}) such that $|f|_{+\infty}=|f|_{-\infty}$, we
define
\begin{equation}\label{vftilde}
 \tilde v_{f} (t,x)= v_{f}(t,x)+ e^{i\frac{\alpha}{2}\log t}\left(e^{it\partial^{2}_{x}} u_{+}\right)(x),
\end{equation}
where
\begin{equation}
 \label{vftilde1}
 v_{f}(t,x)=\bar f \left( \frac{x}{\sqrt{t}} \right), \qquad
 {\hbox{and}}\quad
 \alpha=2|f|^2_{\infty}-A
\end{equation}
(recall that if $|f|_{+\infty}=|f|_{-\infty}$, then we write
$|f|_{\infty}$ for  $|f|_{\pm\infty}$).

In order to prove the existence of a solution $v$ of
\begin{equation}
 \label{eqvF}
 iv_{t}+v_{xx}+\frac{v}{2t}(|v|^2-A)=0,
\end{equation}
``close" to $\tilde v_{f}$, as $t$ goes to $\infty$, following the
steps in the previous lines, we write
\begin{equation}
 \label{change1}
 v=v_f + e^{i\frac{\alpha}{2}\log t}u,
\end{equation}
so that the function $u$ has to be a solution of is
\begin{multline}
 \label{equ}
   iu_{t}+ u_{xx}+
 \frac{1}{t} (|v_f|^2-|f|^2_{\infty})u +
 \frac{v_f^2}{2t}e^{-i\alpha \log t} \bar u +
 \\
 \frac{1}{2t}
 \left[
 2 v_{f} e^{-i\frac{\alpha}{2}\log t}|u|^2 + \bar v_{f} e^{i\frac{\alpha}{2} \log t}u^2+|u|^2 u
 \right]=0.
\end{multline}
Now, notice that under the change of variables (\ref{change1}) and
the definition of $\tilde v_f$ in (\ref{vftilde}), we have that
\begin{eqnarray*}
 v-\tilde v_f
 &=&
 (v_f+ e^{i\frac{\alpha}{2}\log t}u)- (v_f+ e^{i\frac{\alpha}{2}\log t} (e^{it\partial_x^2} u_{+}))
    \\
 &=&
 e^{i\frac{\alpha}{2}\log t} (u-e^{it\partial_{x}^{2}} u_{+}).
\end{eqnarray*}
Therefore, we are reduced to prove the existence of a solution of
(\ref{equ}) ``close" to $e^{it\partial_{x}^{2}} u_{+}$. To this end,
it is convenient to perform a further change of variables.
Precisely, we write
\begin{equation}\label{change2}
 u(t,x)= z(t,x)+ z_{+}(t,x),
 \qquad {\hbox{with}}\qquad z_{+}(t,x)= e^{it\partial_{x}^2} u_{+}.
\end{equation}
Under the change of variable given by (\ref{change2}), equation
(\ref{equ}) becomes
\begin{equation}\label{eqz}
 iz_t +z_{xx}=\frac{1}{2t} \left\{ F_0(z_{+}) -F_{1}(z)-NLT(z+z_{+})\right\}
\end{equation}
where $F_{0}$, $F_1$ and $NLT$ are defined by
\begin{equation} \label{source}
  F_{0}(z_{+})= 2(|v_f|^2-|f|^2_{\infty})z_{+}
             +v_f^2 e^{-i\alpha\log t} \bar z_{+},
\end{equation}
\begin{equation}\label{linear}
 F_1(z)=2(|v_f|^2-|f|^2_{\infty})z +v_f^2 e^{-i\alpha\log t}\bar z,
\end{equation}
and
\begin{equation}\label{NLT}
 NLT(u)=
 2v_f e^{-i\frac{\alpha}{2}\log t} |u|^2 +
 \bar v_f e^{i\frac{\alpha}{2}\log t}u^2+|u|^2u.
\end{equation}
Hence, it suffices to prove the existence of a fixed point of the
operator
\begin{eqnarray}\label{opB}
 Bz(t)
 &=&
 \frac{i}{2}\int_{t}^{\infty} e^{i(t-\tau)\partial_x^2} F_{0}(z_{+})\, d\tau
 - \frac{i}{2}\int_{t}^{\infty} e^{i(t-\tau)\partial_x^2} F_1(z)\, \frac{d\tau}{\tau}
  \nonumber \\
 &-&
 \frac{i}{2}\int_{t}^{\infty} e^{i(t-\tau)\partial_x^2} NLT(z+z_{+})\, \frac{d\tau}{\tau}
\end{eqnarray}
in an appropriate space.

The Duhamel terms which determine the behaviour of the operator $B$
are the ones related to $F_{0}(z_{+})$, the source term. Notice that
in our case we are left to estimate two linear source terms
(see~(\ref{source})). Namely, we need to estimate in an appropriate
space
$$
 \int_{t}^{\infty} e^{i(t-\tau)\partial_{x}^{2}}
 \left(
 (|v_f|^2-|f|^2_{\infty}) z_{+}
 \right)\, \frac{d\tau}{\tau}
$$
and
\begin{equation}
 \label{Duhamel}
 \int_{t}^{\infty} e^{i(t-\tau)\partial_{x}^{2}}
 (v^2_f e^{-i\alpha \log t} \bar z_{+})\, \frac{d\tau}{\tau},
 \qquad {\hbox{with}}\qquad z_{+}=e^{it\partial_x^2} u_{+}.
\end{equation}
The structure of $|v_f(t,x)|-|f|^2_{\infty}$ allows us to treat the
first term as an ``error" term. The second linear term has the extra
difficulty of being dependent of the spatial variable $x$ through
the function $v_f(t,x)=\bar f(x/\sqrt{t})$. In order to estimate the
linear term involving $v^2_f$, we will use some known properties of
the function $v_f$ (more precisely, in our arguments we will make
use of the asymptotic behaviour as $x\rightarrow \infty$ of $v_f$).
Recall that the function $v_f$ is defined by
$$
 v_{f}(t,x)=\bar f \left(  \frac{x}{\sqrt{t}}\right),
$$
where $f$ is any given solution of the equation
\begin{equation}
 \label{eqf1}
 f''+ i \frac{x}{2} f' + \frac{f}{2}\, (|f|^2-A)=0,
 \qquad A\in \mathbb{R}.
\end{equation}
Equation (\ref{eqf1}) was previously considered in~\cite{GV}. The following result summarizes some of the properties
of the solutions $f$ of (\ref{eqf1}) obtained in the latter paper.
%
%----------------- Properties of f -------------------------
%
\begin{proposition}\label{f}\
Let $f$ be a solution of the equation (\ref{eqf1}). Then
\begin{itemize}
\item[{\it i)}] $f(x)$, and $f'(x)$ are bounded globally defined
functions. Moreover, there exists $E(0)> 0$ such that the
identity
$$
 |f'|^2+\frac{1}{4}(|f|^2-A)^2=E(0)
$$
holds true for all $x\in\R$.
\item[{\it ii)}] The limits
$\lim_{x\rightarrow \pm\infty}|f|^2(x)=|f|^2_{\pm\infty}$ and
$\lim_{x\rightarrow \pm\infty}|f'|^2(x)=|f'|^2_{\pm\infty}$ do
exist and
$$
 |f(x)|^2 -|f|^{2}_{\pm\infty}=O\left(  \frac{1}{|x|}   \right),
 \qquad {\hbox{as}}\qquad x\rightarrow \pm \infty.
$$
 \item[{\it {iii)}}] If $|f|_{+\infty}\neq 0$ or
 $|f|_{-\infty}\neq 0$, then
 $$
  f(x)=
  |f|_{\pm\infty}\, e^{ic_\pm}\, e^{i\ps}
  + 2i\, |f'|_{\pm\infty} \,
  {\frac{e^{id_\pm}}{x}}\, e^{i\pt}+ O\left( {\frac{1}{|x|^2}} \right),
  $$
  as $x\to\pm\infty$.
 \end{itemize}
 Here, $|f|_{\pm\infty}$, $|f'|_{\pm\infty}\geq 0$,
 and $c_{\pm}$ and $\ d_{\pm}$ are arbitrary
 constants in $[0,2\pi)$,
 $$
 \ps= (|f|^2_{\pm\infty}-A)\log |x|,\qquad {\hbox{and}}\qquad
 \pt= -(x^2/ 4)- (2|f|^2_{\pm\infty}-A)\log |x|.
 $$
\end{proposition}
%
%-------------------------------------------------------
%
We continue to recall the one-dimensional dispersive and Strichartz
estimates which will used throughout this section (see~\cite{Caz}).
In what follows, we call a pair $(p,q)$ of exponents
{\it{admissible}} if
     $$
      p\geq 2, \qquad q\leq \infty,
      \qquad {\hbox {and}}\qquad
      \frac{2}{p}+\frac{1}{q}= \frac{1}{2}.
     $$
\smallskip

\noindent {\it {i)}} {\sl {1d-Strichartz estimates. }} Let $I$ be a
time interval, then
 \begin{equation}
  \label{stri-h}
  \| e^{it\partial_{x}^{2}} f \|_{L^{p_1}\left( \mathbb{R};L^{q_1}  \right)}
  \leq C\, \| f \|_{L^2(\mathbb{R})}
 \end{equation}
 and
 \begin{equation}
  \label{stri-inh}
  \left\| \int_{s\in I; s\leq t} e^{i(t-s)\partial_{x}^{2}} F(s)\, ds \right\|_{L^{p_1}\left( I;L^{q_1}  \right)}
  \leq C\, \| F \|_{L^{p_2'}\left( I;L^{q_2'} \right)},
 \end{equation}
 for any admissible exponents $(p_i, q_i)$, $i\in \{ 1,2  \}$.
\smallskip

\noindent {\it {ii)}} {\sl {Dispersion estimate. }}
 \begin{equation}
  \label{dis}
  {\| e^{it\partial_{x}^{2}} f\|}_{L^{\infty}(\mathbb{R})}\leq C\, t^{-1/2} {\|f \|}_{L^{1}(\mathbb{R})}
 \end{equation}
\smallskip

\noindent {\it {iii)}} {\sl {$L^2$-Conservation law. }}
 \begin{equation}
  \label{L2}
  {\| e^{it\partial_{x}^{2}} f\|}_{L^{2}(\mathbb{R})}= {\|f \|}_{L^{2}(\mathbb{R})}
 \end{equation}
The constant $C$ in the above inequalities depend on the exponents
involved in the estimates.

%------------------- Some lemmas --------------------------------
%
The lemmas below will be also used in the construction of the
modified wave operators. It is immediate to prove the following:
\begin{lemma}
 \label{lemma1}
 Let $0\leq \beta\leq 4$, $f\in L^2(|x|^\beta)$, and $t>0$. Then
$$
 {\|
  f(\cdot) (e^{-i\frac{(\cdot)^2}{4t}}-1)
 \|}_{L^2}\leq \frac{C}{t^{\beta/4}}
 {\| f\|}_{L^2(|x|^{\beta})},
$$
for some positive constant $C$ independent of $f$ and $t$.
\end{lemma}
\begin{lemma}
 \label{lemma2}
 Given $\delta\neq 0$, and $t>0$, define
 $$
   A_{t}(\xi)=\int_{t}^{\infty} e^{2i\tau\xi^2} \frac{d\tau}{\tau^{1+i\delta}}=
   \int_{t}^{\infty} e^{2i\tau\xi^2} e^{-i\delta\log\tau} \frac{d\tau}{\tau}
 $$
 for $\xi\neq 0$. Then, there exists a constant $C>0$ such that
\begin{equation}
\label{lemma2a}
 |A_{t}(\xi)|\leq \frac{C}{1+t\xi^2}, \qquad \forall\, \xi\neq 0.
\end{equation}
\end{lemma}
\begin{proof}
 For fixed $\xi\neq 0$ , and $t>0$ such that $\xi^2t\geq 1$, write
 $$
  A_{t}(\xi)= \frac{1}{2i\xi^2} \int_{t}^{\infty}
   \frac{d\ }{d\tau}(e^{2i\tau\xi^2})\, \frac{d\tau}{\tau^{1+i\delta}},
 $$
 and for $\xi$ and $t$ such that $\xi^2t\leq 1$, write
 \begin{eqnarray*}
 A_{t}(\xi)
 &=&
 \left(\int_{t}^{1/\xi^2} + \int_{1/\xi^2}^{\infty}    \right)
 e^{2i\tau\xi^2}\ e^{-i\delta\log \tau}, \frac{d\tau}{\tau}
   \\
 &=&
  \frac{i}{\delta}\int_{t}^{1/\xi^2} \frac{d\ }{d\tau} (e^{-i\delta\log\tau}) e^{2i\tau\xi^2}\, d\tau
  +\frac{1}{2i\xi^2} \int_{1/\xi^2}^{\infty}
  \frac{d\ }{d\tau}(e^{2i\tau\xi^2})\, \frac{d\tau}{\tau^{1+i\delta}}.
 \end{eqnarray*}
 Inequality (\ref{lemma2a}) now follows by integrating by parts in
 the above identities.
\end{proof}
\begin{lemma}[Pitt's inequality. See~\cite{Pitt}]
 \label{lemma3}
 For $f\in\mathcal{S}(\R^d)$, and $0\leq \beta<d$,
 $$
 \int_{\R^d} |\xi |^{-\beta} |\hat f(\xi)|^2\, d\xi \leq
 C_{\beta} \int_{\R^d} |x|^{\beta} |f(x)|^2\, dx,
 $$
 where
 $$
  C_{\beta}= \pi^{\beta}\left[
  \Gamma\left(  \frac{d-\alpha}{4} \right)\Big/ \Gamma \left( \frac{d+\alpha}{4}  \right)
  \right]^{2}
 $$
\end{lemma}

We will continue to prove Theorem~\ref{T1}.
%
%--------------------------------------------------------
\subsection{Proof of Theorem~\ref{T1}}
%--------------------------------------------------------
Let $0<\gamma<1$, and $u_{+}\in L^1\cap L^2(\langle x
\rangle^\gamma)$.

For  $t_0\geq 1$ and $\nu\in \R$. We denote  $Y^{\nu}_{t_0}$ the
space of functions $v(t,x)$ such that the norm
$$
 {\| z \|}_{Y^{\nu}_{t_0}}= \sup_{t\in [t_{0}, \infty)}
 \left(
 t^{\nu} {\| z(t)  \|}_{L^2} +
 t^{\nu} {\| z  \|}_{L^4 \left(  (t,\infty), L^{\infty} (\mathbb{R})
 \right)}
 \right)
$$
is finite. In order to prove Theorem~\ref{T1}, as we have already
mentioned, we shall do a fixed point for the operator $B$ defined in
(\ref{opB}) in the closed ball
$$
 \mathcal{B}_{R}=\{
 z\ : \ {\| z\|}_{Y^{\nu}_{t_0}}\leq R
 \}, \qquad R>0
$$
with $\nu>0$  and $R>0$ to be chosen later on.

For any given $z$ such that\footnote{In order to simplify notation,
in what follows we will write simply $Y$ to denote the space
$Y_{t_{0}}^{\nu}$.} $\|z \|_{Y}\leq R$,  we want to estimate
(see~(\ref{source}), (\ref{linear}), (\ref{NLT}) and \ref{opB}))
\begin{eqnarray}
 \label{b0}
 (Bz)(t)
 &=& \frac{i}{2}\int_{t}^{\infty} e^{i(t-\tau)\partial_x^2}\left\{
 2 (|v_f|^2-|f|^2_{\infty})z_{+}+ v_f^2 e^{-i\alpha \log t } \bar z_{+}
 \right\}\, \frac{d\tau}{\tau}
    \nonumber \\
 &-&
 \frac{i}{2}\int_{t}^{\infty} e^{i(t-\tau)\partial_{x}^{2}}
 \left\{
 2(|v_{f}|^2 -|f|^2_{\infty}) z
 +v_{f}^{2} e^{-i\alpha \log \tau} \bar z
 \right\}\frac{d\tau}{\tau}
    \nonumber \\
 &-&
 \frac{i}{2}\int_{t}^{\infty} e^{i(t-\tau)\partial_{x}^{2}} NLT(z+z_{+})\, \frac{d\tau}{\tau}
\end{eqnarray}
in $Y$.  Here, the non-linear term is given by (see
(\ref{NLT}))
\begin{equation}
 \label{NLTb}
 NLT(u)=
 2v_f e^{-i\frac{\alpha}{2}\log t} |u|^2 +
 \bar v_f e^{i\frac{\alpha}{2}\log t}u^2+|u|^2u.
\end{equation}
Recall that (see (\ref{vftilde1}) and (\ref{change2}))
\begin{eqnarray}
 \label{b1}
v_f(t,x)=\bar f\left( \frac{x}{\sqrt{t}} \right),
\qquad
z_{+}(t,x)= e^{it\partial_x^2} u_{+},
\qquad
\alpha = 2|f|^2-A,
\end{eqnarray}
and $u_{+}$ is a given function in $L^1\cap L^2(\langle x
\rangle^\gamma)$, with $0<\gamma<1$.

In what follows the constant $C$ may be different from an inequality
to another in a chain of inequalities.
%
%-------------Analysis of the source terms ----------------
%

\noindent We begin by estimating the $Y$-norm of the source term in
(\ref{b0}). To this end, observe that Strichartz estimates
(\ref{stri-inh}) with exponents $(4,\infty)$ and $(\infty, 2)$, and
the dispersion estimate (\ref{dis})  lead to
\begin{eqnarray}
 \label{b2}
 &&
 {\left\|
 \int_{t}^{\infty} e^{i(t-\tau)\partial_{x}^{2} }
 \left(
 (|v_{f}|^2- |f|^{2}_{\infty}) z_{+}
 \right)\,
 \frac{d\tau}{\tau}
 \right\|}_{Y}
    \nonumber  \\
 &&
 \hspace{1truecm}
 \leq
 C \sup_{t\geq t_{0}} t^{\nu}
 \int_{t}^{\infty}
 {\|  (|v_{f}|^2-|f|^2_{\infty}) z_{+} \|}_{L^2} \, \frac{d\tau}{\tau}
    \nonumber \\
 &&
 \hspace{1truecm}
 \leq
 C \sup_{t\geq t_{0}} t^{\nu}
 \int_{t}^{\infty}
 {\|  |v_{f}|^2-|f|^2_{\infty}\|}_{L^2}{\| z_{+} \|}_{L^\infty} \,
 \frac{d\tau}{\tau}
    \nonumber \\
 &&
 \hspace{1truecm}
 \leq
 C {\|u_{+}\|}_{L^1}
 {\| \, |f(\cdot)|^2-|f|^{2}_{\infty}  \|}_{L^2}
 \sup_{t\geq t_{0}} t^{\nu} \int_{t}^{\infty}
 \, \frac{d\tau}{\tau^{1+\frac{1}{4}}}
    \nonumber  \\
 &&
 \hspace{1truecm}
 =
 C {\|u_{+}\|}_{L^1}
 {\| \, |f(\cdot)|^2-|f|^{2}_{\infty}  \|}_{L^2}
 \, \frac{1}{{t_0}^{\frac{1}{4}-\nu}},
\end{eqnarray}
for all $\nu$ such that $\nu\leq1/4$.
\medskip

%
%------------------ Second source term: The difficult one --------------------
%
In order to control the second source term in (\ref{b0}), we use the
fact that $f$ is a solution of (\ref{eqf1}) satisfying
$|f|_{+\infty}=|f|_{-\infty}$ (and as a consequence
$|f'|_{+\infty}=|f'|_{-\infty}$, see part {\it{i)}} in
Proposition~\ref{f}). Then, from the asymptotics of $f$ in
Proposition~\ref{f}, it follows that
\begin{eqnarray}\label{b3}
 (\bar f)^2(x)
 &=&
 |f|^{2}_{\infty} e^{-2ic_{\pm}} e^{-2i(|f|^2_{\infty}-A)\log|x|}
    \nonumber  \\
 &-&
 4i |f|_{\infty} |f'|_{\infty} e^{-i(c_{\pm}+ d_{\pm})} \frac{e^{i \left(  \frac{x^2}{4}+|f|^2_{\infty}\log |x|  \right)}}{x}+
 O\left(  \frac{1}{|x|^2}  \right),
\end{eqnarray}
as $x\rightarrow \pm \infty$, so that
\begin{equation}
 \label{b4}
 v_{f}^{2}(t,x)=(\bar f)^2\left( \frac{x}{\sqrt{t}} \right)=
 v^2_{f,\infty}(t,x) +O\left( \frac{\sqrt{t}}{|x|}\right),
\end{equation}
for $|x|\gg \sqrt{t}$, where we denote by $v^2_{f,\infty}(t,x)$ the
function defined for positive time by
\begin{equation}\label{b5}
 v^2_{f,\infty}(t,x)=
 |f|^2_{\infty}   e^{-2i\delta\log\left| \frac{x}{\sqrt{t}}  \right|}\, m(x), \qquad
 m(x)=e^{-2ic_{+}}\chi_{[0,\infty)}(x)+ e^{-2ic_{-}}\chi_{(-\infty,0)}(x)
 %= |f|^2_{\infty} e^{-2ic_{+}} e^{i\delta\log 2} t^{i\frac{\delta}{2}} e^{i\delta \log\left| -\frac{x}{2t}  \right|},
\end{equation}
with $c_{\pm}\in [0,2\pi)$, and $\delta=|f|^{2}_{\infty}-A$.

Next notice that, from the representation of the solution of the
free Schr\"odinger equation, $e^{-i\tau\partial_x^2}u_{0}$, as a
convolution, it is easy to see that
\begin{eqnarray}
 \label{b6}
  \left( e^{-i\tau\partial_x^2}u_0  \right)(x)
  &=&
  \overline{\frac{1}{\sqrt{4\pi i\tau}}}\int u_0(y) e^{-i\frac{(x-y)^2}{4\tau}}\, dy
  \nonumber  \\
  &=&
  \frac{c_1}{\sqrt{\tau}} e^{-i\frac{x^2}{4\tau}} (\widehat{u_0})\left(- \frac{x}{2\tau}  \right)
  \nonumber  \\
  &+&
  \frac{c_1}{\sqrt{\tau}} e^{-i\frac{x^2}{4\tau}}
   \left( u_0(\cdot) (e^{-\frac{(\cdot)^2}{4\tau}}-1)  \right)^{\widehat{\ }} \left(-\frac{x}{2\tau}
   \right),
\end{eqnarray}
with $c_1=\sqrt{\pi i}$.

Define the auxiliary function $\omega_{+}$ by
\begin{equation}\label{b7}
\omega_{+}= e^{i\alpha\log t} u_{+}.
\end{equation}
Then, from (\ref{b1}), (\ref{b7}), (\ref{b4}) and (\ref{b6}), it is
straightforward to see that the second linear term in (\ref{b0}) can
be  rewritten as
\begin{eqnarray} \label{b8}
 &&
 \int_{t}^{\infty} e^{i(t-\tau)\partial_{x}^{2}} ( v^2_{f} e^{-i\alpha \log t} \bar z_{+} )\, \frac{d\tau}{\tau}
 =
 \int_{t}^{\infty} e^{i(t-\tau)\partial_{x}^{2}} ( v^2_{f} e^{-i\tau \partial_{x}^{2}} \overline{\omega}_{+} )\, \frac{d\tau}{\tau}
 \nonumber \\
 &&
 %{\hspace{0.5truecm}}
 =
 \int_{t}^{\infty} e^{i(t-\tau)\partial_{x}^{2}}
 \left[ (v_{f}^2-v^{2}_{f,\infty})
 %|f|^{2}_{\infty} e^{-2ic_{+}} e^{i\delta\log\left|  \frac{x}{\sqrt{\tau}} \right|})\,
 e^{-i\tau \partial_{x}^{2}} \overline{\omega}_{+} \right]
 \, \frac{d\tau}{\tau}
    \nonumber \\
 &&
 %{\hspace{0.5truecm}}
 \qquad +
 \int_{t}^{\infty} e^{i(t-\tau)\partial_{x}^{2}}
 (
 v^2_{f,\infty}e^{-i\tau \partial_{x}^{2}} \overline{\omega}_{+}
 )\, \frac{d\tau}{\tau}
  \nonumber \\
 &&
 %{\hspace{0.5truecm}}
 =
 \int_{t}^{\infty} e^{i(t-\tau)\partial_{x}^{2}}
 \left[ (v_{f}^2-v^{2}_{f,\infty})
 %|f|^{2}_{\infty} e^{-2ic_{+}} e^{i\delta\log\left|  \frac{x}{\sqrt{\tau}} \right|})\,
 e^{-i\tau \partial_{x}^{2}} \overline{\omega}_{+} \right]
 \, \frac{d\tau}{\tau}
    \nonumber \\
 &&
 %{\hspace{0.5truecm}}
 \qquad +
 c_1 \int_{t}^{\infty} e^{i(t-\tau)\partial_{x}^{2}}
 \left[
  v^2_{f,\infty} \frac{e^{-i\frac{x^2}{4\tau}}}{\sqrt{\tau}}
  \left( \overline{\omega}_{+}(\cdot) (e^{-i\frac{(\cdot)^2}{4\tau}}-1) \right)^{\widehat{\ }} \left( -\frac{x}{2\tau} \right)
 \right] \, \frac{d\tau}{\tau}
     \nonumber  \\
 &&
 %{\hspace{0.5truecm}}
 \qquad +
 c_1 \int_{t}^{\infty} e^{i(t-\tau)\partial_{x}^{2}}
 \left[
 v^2_{f,\infty}
 \frac{e^{-i\frac{x^2}{4\tau}}}{\sqrt{\tau}}(\widehat{\overline{\omega}_{+}}) \left( -\frac{x}{2\tau} \right)  %\widehat{\bar{\omega_{+}}}
 \right] \, \frac{d\tau}{\tau}.
 %    \nonumber \\
 %&&
 %{\hspace{0.5truecm}}
 %=
 %I_{1}+I_{2}+I_{3}.
\end{eqnarray}
To control the first term on the r.h.s. in (\ref{b8}), we first
observe that
$$
{\|
  v_{f}^2(\tau, \cdot)-v^2_{f,\infty}(\tau,\cdot)\,
\|}_{L^2(\mathbb{R})}=
\tau^{\frac{1}{4}}
{\| (\bar f)^2(\cdot) -  (\bar f)^2_{\infty}(\cdot)  \|}_{L^2(\mathbb{R})},
$$
where
%\begin{equation} \label{}
 $$
 (\bar f)^2_{\infty}(x)=|f|^2_{\infty}  e^{-2i\delta \log|x|} m(x),
\qquad \delta=|f|^2_{\infty}-A,
 $$
%\end{equation}
and ${\| (\bar f)^2(\cdot) -(\bar f)^2_{\infty}(\cdot)
 \|}_{L^2}<\infty$ (recall that $f$ is a bounded function and the
asymptotics of $(\bar f)^2(x)$ given in (\ref{b3})). Then, arguing
similarly to the control of the first linear term (see~(\ref{b2})),
we obtain that
\begin{eqnarray}\label{b9}
 &&
 \left\|
 \int_{t}^{\infty} e^{i(t-\tau)\partial_{x}^{2}}
 \left(
 ( v_{f}^2- v_{f,\infty}^2
 %|f|^{2}_{\infty} e^{-2ic_{+}} e^{i\delta\log\left|  \frac{x}{\sqrt{\tau}} \right|}\,
 )e^{-i\tau \partial_{x}^{2}} \overline{\omega}_{+}
 \right)
 \, \frac{d\tau}{\tau}
 \right\|_{Y}
    \nonumber   \\
 &&
 \hspace{0.5truecm}
 \leq
 C\, {\|u_{+}\|}_{L^1}  {\| (\bar f)^2(\cdot)- (\bar f)^2_{\infty}(\cdot)   \|}_{L^2}
 \frac{1}{{t_0}^{\frac{1}{4}-\nu}},
\end{eqnarray}
for $\nu\leq 1/4$.

The second term in the r.h.s. in (\ref{b8}) is an error term.
Strichartz estimates with exponents $(4,\infty)$ and $(\infty, 2)$,
(\ref{b5}), Plancherel's identity, and lemma~\ref{lemma1} lead to
\begin{eqnarray}
 \label{b10}
 &&
 {\left\|
 \int_{t}^{\infty} e^{i(t-\tau)\partial^2_{x}}\left[
 v^2_{f,\infty}(\cdot,\tau) \frac{e^{-i\frac{x^2}{4\tau}}}{\sqrt{\tau}}
 (\overline{\omega}_{+} (\cdot)(e^{-i\frac{(\cdot)^2}{4\tau}}-1))^{\widehat{\ }}
 \left(-\frac{x}{2\tau}  \right)
 \right]\, \frac{d\tau}{\tau}
 \right\|}_{Y}
   \nonumber   \\
 &&
 \hspace{0.5truecm}
 \leq
 C\,
 \sup_{t\geq t_0} t^{\nu} \int_{t}^{\infty}
 {\left\|
 v^2_{f,\infty}(x,\tau)  \frac{e^{-i\frac{x^2}{4\tau}}}{\sqrt{\tau}}
 \left(\overline{\omega}_{+}(\cdot) (e^{-i\frac{(\cdot)^2}{4\tau}}-1) \right)^{\widehat{\ }} \left( -\frac{x}{2\tau} \right)
 \right\|}_{L^2}
 \, \frac{d\tau}{\tau}
    \nonumber  \\
 &&
 \hspace{0.5truecm}
 =
 C\, |f|^2_{\infty}
 \sup_{t\geq t_0} t^{\nu} \int_{t}^{\infty}
 {\left\|
 \left(
 \overline{\omega}_{+}(\cdot) (e^{-i\frac{(\cdot)^2}{4\tau}}-1)
 \right)^{\widehat{\ }} \left( -\frac{x}{2\tau} \right)
 \right\|}_{L^2}
 \, \frac{d\tau}{{\tau}^{\frac{3}{2}}}
    \nonumber  \\
 &&
 \hspace{0.5truecm}
 =
 C\, |f|^2_{\infty}
 \sup_{t\geq t_0} t^{\nu} \int_{t}^{\infty}
 {\left\|
 \overline{\omega}_{+}\,  (e^{-i\frac{(\cdot)^2}{4\tau}}-1)
 \right\|}_{L^2}
 \, \frac{d\tau}{\tau}
    \nonumber  \\
 &&
 \hspace{0.5truecm}
 \leq
 C\, {\|f\|}^{2}_{L^{\infty}}\, {\|u_{+}\|}_{L^2(|x|^{\gamma})}
 \, \frac{1}{{t_0}^{\frac{\gamma}{4}-\nu}},
\end{eqnarray}
for $0<\gamma\leq 4$ and $\nu\leq \gamma/4$, (recall that
$\omega_{+}=e^{i\alpha\log t}u_{+}$, see (\ref{b7})).

In order to control the third term on the r.h.s. in (\ref{b8}),
recall the definition of $v_{f,\infty}^{2}(t,x)$ in (\ref{b5}),
$$
  v^2_{f,\infty}(t,x)=
 |f|^2_{\infty}  e^{-2i\delta\log\left| \frac{x}{\sqrt{t}}  \right|} m(x),\qquad
 m(x)= e^{-2ic_{+}}\chi_{[0,\infty)}(x)+e^{-2ic_{-}}\chi_{(-\infty,0)}(x)
$$
or, equivalently,
$$
v_{f,\infty}^{2}(t,x)
 = |f|^2_{\infty}  e^{-2i\delta\log 2} t^{-i\delta} e^{-2i\delta \log\left| \frac{x}{2t}  \right|} m\left(  \frac{x}{2t} \right), \qquad t>0.
$$
Then, using once again the expression for the free Shr\"odinger
solution in (\ref{b6}), the latter term rewrites
\begin{eqnarray}
 \label{b11}
 &&
 c_1 \int_{t}^{\infty} e^{i(t-\tau)\partial^2_{x}}\left(
 v^{2}_{f,\infty} \frac{e^{-i\frac{x^2}{4\tau}}}{\sqrt{\tau}}
 (\widehat{\overline{\omega}_{+}})\left( -\frac{x}{2\tau} \right)
 \right)\, \frac{d\tau}{\tau}
    \nonumber \\
 &&\hspace{0.2truecm}
 =|f|^2_{\infty} e^{-2i\delta\log 2}
 \int_{t}^{\infty} e^{i(t-\tau)\partial_x^2}\left(
 c_1 \frac{e^{-i\frac{x^2}{4\tau}}}{\sqrt{\tau}} e^{-2i\delta \log\left|\frac{x}{2\tau}  \right|}
 m\left( \frac{x}{2\tau} \right)(\widehat{\overline{\omega}_{+}})\left( -\frac{x}{2\tau} \right)
 \right) \, \frac{d\tau}{\tau^{1+i\delta}}
   \nonumber  \\
 &&\hspace{0.2truecm}
 =|f|^2_{\infty} e^{-2i\delta\log 2}
 \int_{t}^{\infty} e^{i(t-\tau)\partial_x^2}\left(
 c_1 \frac{e^{-i\frac{x^2}{4\tau}}}{\sqrt{\tau}}
  (\widehat{\overline{T_{\delta} \omega_{+} }})
  \left( -\frac{x}{2\tau} \right)
 \right) \, \frac{d\tau}{\tau^{1+i\delta}}
   \nonumber  \\
 &&\hspace{0.2truecm}
 =|f|^2_{\infty} e^{-2i\delta\log 2}
 \int_{t}^{\infty} e^{i(t-2\tau)\partial_x^2}\left(
 \overline{T_{\delta} \omega_{+}} \right)
 \, \frac{d\tau}{\tau^{1+i\delta}}
   \nonumber  \\
 &&\hspace{0.2truecm}
 \quad
 -c_1 |f|^2_{\infty}  e^{-2i\delta\log 2}
 \int_{t}^{\infty} e^{i(t-\tau)\partial_x^2}
   \nonumber \\
 &&\hspace{3.0truecm}
 \left(
 e^{-i\frac{x^2}{4\tau}}
 \left(
 ( \overline{T_{\delta}\omega_{+}} )(\cdot ) (e^{-\frac{(\cdot)^2}{4\tau}}-1)
 \right)^{\widehat{\ }}
 \left( -\frac{x}{2\tau}
 \right)
 \right)\,
 \frac{d\tau}{\tau^{\frac{3}{2}+i\delta}}
   \nonumber  \\
 &&\hspace{0.2truecm}
 = I_1+I_2.
\end{eqnarray}
Here $T_{\delta}$ is the operator defined (in the Fourier
transform side) by
\begin{equation}
 \label{b12}
 \widehat{T_{\delta}u}(\xi)= e^{2i\delta \log|\xi|} \bar m(\xi) \hat u (\xi).
\end{equation}
$I_2$ is an ``{\it{error}}" term. The same argument as the one given
in obtaining (\ref{b10}) (that is using Strichartz estimates with
exponents $(4,\infty)$ and $(\infty, 2)$, Plancherel's identity, and
Lemma~\ref{lemma1}) leads to the following chain of inequalities
\begin{eqnarray}
 \label{b13}
 {\|I_2\|}_{Y}
 &\leq&
 C |f|^2_{\infty}\sup_{t\geq t_0} t^{\nu} \int_{t}^{\infty}
 {\| (   \overline{T_{\delta} \omega_{+}}  (\cdot )(e^{-\frac{(\cdot)^2}{4\tau}}-1))^{\widehat{\ }}
 \left(  -\frac{x}{2\tau} \right)
 \|}_{L^2}
 \, \frac{d\tau}{\tau^{1+\frac{1}{2}}}
    \nonumber  \\
 &\leq&
  C |f|^2_{\infty}\sup_{t\geq t_0} t^{\nu} \int_{t}^{\infty}
 {\| (\overline{T_{\delta}\omega_{+}}  )(e^{-\frac{(\cdot)^2}{4\tau}}-1)
 \|}_{L^2}
 \, \frac{d\tau}{\tau}
    \nonumber  \\
 &\leq&
 C |f|^2_{\infty}\sup_{t\geq t_0} t^{\nu} \int_{t}^{\infty}
 {\|T_{\delta}\omega_{+} \|}_{L^2(|x|^{\gamma})}
 \, \frac{d\tau}{{\tau}^{1+\frac{\gamma}{4}}},
\end{eqnarray}
for any $0\leq\gamma\leq 4$.

Now, since $T_{\delta}=T_1\circ T_2$, where $T_1$ and $T_2$ are
defined by
\begin{eqnarray*}
 &&
 \widehat{T_1 f}(\xi)= e^{2i\delta\log|x|}\, \hat f (\xi)
 \qquad {\hbox{and}}\qquad\\
 &&
 \widehat{T_2 f}(\xi)=\bar m(\xi) \hat f(\xi)=
 \left(
 \frac{e^{2ic_{+}}}{2}(1+\sgn(\xi)) +
 \frac{e^{2ic_{-}}}{2}(1-\sgn(\xi))
 \right)
 \hat f(\xi)
\end{eqnarray*}
with $T_1$ and $T_2$ Calder\'on-Zygmund operators (see
\cite[pp.~97-98]{Duo} and notice that $T_2$ is just a linear
combination of the identity operator and the Hilbert transform), and
$|x|^{\gamma}$ is an $A_2$-weight in the one-dimensional case, in
particular, for any $0\leq \gamma<1$, from the known $L^2$-weighted
inequalities  for Calder\'on-Zygmund operators (see \cite[pp.~
144]{Duo}, or \cite[pp.~204-205]{Stein}), we have that
\begin{equation}
 \label{b14}
 {\| T_{\delta} u \|}_{L^2(|x|^{\gamma})}\leq C {\| u \|}_{L^2(|x|^{\gamma})},
\end{equation}
for any $0\leq \gamma <1$.

From the inequalities (\ref{b13}) and (\ref{b14}), we conclude that
\begin{equation}
 \label{b15}
 {\|  I_2\|}_{Y}\leq \frac{C}{t_{0}^{\frac{\gamma}{4}-\nu}}
 {\| f \|}_{L^{\infty}}^{2}
 {\|  \omega \|}_{L^{2}(|x|^{\gamma})},
\end{equation}
for any $0<\gamma<1$, and $\nu\leq \gamma/4$.

Only $I_1$ remains to be estimated. First, recall
(\ref{sh-solution}), and the definition of $T_{\delta}$ in
(\ref{b12}). Then, $I_{1}$ in (\ref{b11}) rewrites equivalently as
\begin{eqnarray}
 \label{b16}
 I_{1}
 &=&
 |f|^{2}_{\infty}  e^{-2i\delta \log 2} \int_{t}^{\infty}
 e^{i(t-2\tau)\partial_{x}^{2}} (\overline{T_{\delta} \omega_{+}})
 \, \frac{d\tau}{\tau^{1+i\delta}}
  \nonumber  \\
 &=&
 |f|^{2}_{\infty} e^{-2i\delta \log 2} \int_{t}^{\infty}
 \left(
 \int_{\R} e^{ix\xi}e^{-i(t-2\tau)\xi^{2}} \widehat{\overline{T_{\delta}\omega_{+}}}(\xi)\, d\xi
 \right)
 \, \frac{d\tau}{\tau^{1+i\delta}}
  \nonumber  \\
 &=&
 |f|^{2}_{\infty} e^{-2i\delta \log 2}
 \int_{\R}
 e^{ix\xi} e^{-it\xi^2} e^{-2i\delta \log|\xi|} m(-\xi) \widehat{\overline{\omega}_{+}}(\xi)\, A_{t}(\xi)\, d\xi,
\end{eqnarray}
where
$$
 A_{t}(\xi)=\int_{t}^{\infty} e^{2i\tau\xi^2} \, \frac{d\tau}{\tau^{1+i\delta}}.
$$
On the one hand, from Plancherel's identity and Lemmma~\ref{lemma2},
it is easy to see that
\begin{eqnarray*}
 {\| I_{1}\|}_{L^2}
 &=&
 |f|^{2}_{\infty} {\| m(-\cdot) \widehat{\overline{\omega}_{+}}(\cdot)\, A_{t}(\cdot)\|}_{L^2}
    \\
 &\leq&
 2 |f|^{2}_{\infty}
 \left(
  {\| \widehat{\overline{\omega}_{+}}\, A_{t}\|}_{L^2(t\xi^2\leq 1)}+
 {\| \widehat{\overline{\omega}_{+}}\, A_{t}\|}_{L^2(t\xi^2> 1)}
 \right)
    \\
 &\leq&
 C\, \frac{|f|^{2}_{\infty}}{t^{\gamma/4}}
 \left(
 \int \frac{|\widehat{\overline{\omega}_{+}}(\xi)|^2}{|\xi|^{\gamma}}\, d\xi
 \right)^{1/2}
   \\
 &=&
 C\, \frac{\|f\|^{2}_{L^\infty}}{t^{\gamma/4}}
 {\| \omega_{+} \|}_{L^2(|x|^{\gamma})},
\end{eqnarray*}
for any $0\leq \gamma<1$. Here, we have used Pitt's inequality (see
Lemma~\ref{lemma3}) to obtain the last inequality. Thus,
\begin{equation}
 \label{b17}
 \sup_{t\geq t_0} t^{\nu} {\| I_{1}\|}_{L^2}
 \leq
 \frac{C}{t_{0}^{\frac{\gamma}{4}-\nu}} {\| f\|}^{2}_{L^{\infty}} {\| \omega_{+}\|}_{L^2(|x|^{\gamma})},
\end{equation}
for any $0\leq \gamma<1$ and $\nu\leq \gamma/4$.
\smallskip

\noindent In order to estimate the $L^4((t,\infty),
L^{\infty}(\R))$-norm of $I_{1}$, consider $\theta$ a cut-off
function with $\theta(x)=0$ if $|x|\leq 1/2$, and $\theta(x)=1$, if
$|x|>1$. We decompose $I_1$ in (\ref{b16}) as follows
\begin{eqnarray}
 \label{b18}
 I_{1}
 &=&
 |f|^{2}_{\infty}  e^{-2i\delta \log 2}
 \left(
 \int (1-\theta)(t\xi^2)+\int \theta(t\xi^2)
 \right)\,
    \nonumber \\
 &\ &
 \left(
  e^{ix\xi} e^{-it\xi^2} e^{-2i\delta \log|\xi|} m(-\xi)\widehat{\overline{\omega}_{+}}(\xi)\, A_{t}(\xi)
 \right)\, d\xi,
   \nonumber\\
 &=&
 I_{1,1}+I_{1,2}.
\end{eqnarray}
 Using Lemma~\ref{lemma2}, and Cauchy-Schwarz inequality, we find
 that
\begin{eqnarray*}
  |I_{1,2}|
  &\leq&
  2 |f|^{2}_{\infty} \int_{t\xi^2\geq 1/2}
  |\widehat{\overline{\omega}_{+}}(\xi)|\, |A_{t}(\xi)|\, d\xi
  \leq C\, \frac{|f|^{2}_{\infty}}{t} \int_{t\xi^2\geq 1/2} |\widehat{\overline{\omega}_{+}}(\xi)|\, \frac{d\xi}{\xi^2}
    \\
  &=&
  C\, \frac{|f|^{2}_{\infty}}{t} \int_{t\xi^2\geq 1/2} \frac{|\widehat{\overline{\omega}_{+}}(\xi)|}{|\xi|^\frac{\gamma}{2}}\,
  \frac{|\xi|^\frac{\gamma}{2}}{\xi^2}\, d\xi
      \\
  &\leq&
  C\, \frac{|f|^2_{\infty}}{t^{\frac{\gamma}{4}+\frac{1}{4}}}\left(
  \int
  \frac{|\widehat{\overline{\omega}_{+}}(\xi)|^2}{|\xi|^\gamma}\, d\xi
  \right)^{1/2}
\end{eqnarray*}
and
\begin{eqnarray*}
 |I_{1,1}|
 &\leq&
 |f|^{2}_{\infty}
 \int_{t\xi^2\leq 1} |\widehat{\overline{\omega}_{+}}(\xi)|\, |A_{t}(\xi)|\, d\xi
 \leq \frac{C}{t^{\frac{\gamma}{4}}} |f|^{2}_{\infty} \int_{t\xi^2\leq 1} \frac{|\widehat{\overline{\omega}_{+}}(\xi)|}{|\xi|^{\frac{\gamma}{2}}}\, d\xi
   \\
 &\leq&
 \frac{C}{t^{\frac{\gamma}{4}+\frac{1}{4}}} |f|^{2}_{\infty}
 \left(
  \int\frac{| \widehat{\overline{\omega}_{+}}(\xi)|^2}{|\xi|^{\gamma}}\, d\xi\right)^{1/2},
\end{eqnarray*}
for any $0\leq \gamma<3$. Plugging the above inequalities into
(\ref{b18}) and using Pitt's inequality (see~Lemma~\ref{lemma3})
give
$$
 |I_{1}|
 \leq
 C \frac{|f|^{2}_{\infty}}{t^{\frac{\gamma}{4}+\frac{1}{4}}}
 \left(\int\frac{| \widehat{\overline{\omega}_{+}}(\xi)|^2}{|\xi|^{\gamma}}\, d\xi\right)^{1/2}
 \leq
  C \frac{|f|^{2}_{\infty}}{t^{\frac{\gamma}{4}+\frac{1}{4}}} {\|  \omega_{+} \|}_{L^2(|x|^\gamma)},
$$
for $0\leq \gamma <1$. Therefore,
\begin{equation}
 \label{b19}
 \sup_{t\geq t_0} t^{\nu} {\| I_{1}\|}_{L^4((t,\infty), L^{\infty})}
 \leq
 C \frac{|f|^{2}_{\infty}}{t_{0}^{\frac{\gamma}{4}-\nu}}
 \, {\| \omega_{+}\|}_{L^2(|x|^{\gamma})},
\end{equation}
for $0<\gamma<1$, $\nu\leq \gamma/4$.
\medskip

\noindent From (\ref{b17}) and (\ref{b19}), we get that
\begin{equation}
 \label{b20}
 {\| I_{1}\|}_{Y}
 \leq
 C \frac{{\|f\|}^{2}_{L^{\infty}}}{t_{0}^{\frac{\gamma}{4}-\nu}}
 \, {\| \omega_{+}\|}_{L^2(|x|^{\gamma})},
\end{equation}
for $0<\gamma<1$, and $\nu\leq \gamma/4$. Thus, from (\ref{b11}),
(\ref{b15}), and (\ref{b20}), we conclude the following control for
the last term on the r.h.s. in (\ref{b8})
\begin{equation}
 \label{b21}
 \left\|
 \int_{t}^{\infty} e^{i(t-\tau)\partial_{x}^{2}} \left(
 v_{f,\infty}^{2}
 \frac{e^{-i \frac{x^2}{4\tau}}}{\sqrt{\tau}}
 (\widehat{\overline{\omega}_{+}})\left(  -\frac{x}{2\tau} \right)
 \right)\,  \frac{d\tau}{\tau}
 \right\|_{Y}
 \leq
 \frac{C}{t_{0}^{\frac{\gamma}{4}-\nu}} {\| f\|}_{L^{\infty}}^{2}
 \,  {\| u_{+}\|}_{L^2(|x|^{\gamma})},
\end{equation}
for any $\nu\leq \frac{\gamma}{4}$, and $0< \gamma <1$. Recall that
$\omega_{+}= e^{i\alpha \log t}u_{+}$ (see~(\ref{b7})).

Finally, the identity (\ref{b8}), and the inequalities (\ref{b9}),
(\ref{b10}) and (\ref{b21}) give the following control of the second
source term in (\ref{b0})
\begin{eqnarray}
 \label{b22}
 &&
 {\left\|
 \int_{t}^{\infty} e^{i(t-\tau)\partial_{x}^{2}} (v_{f}^{2} e^{-i\alpha \log t}\bar z_{+})\, \frac{d\tau}{\tau}
 \right\|}_{Y}
 \leq
 \frac{C}{t_{0}^{\frac{1}{4}-\nu}} {\| u_{+}\|}_{L^1} {\| (\bar f)^2(\cdot) -(\bar f)^2_{\infty}(\cdot)  \|}_{L^2}
    \nonumber \\
 &&\hspace{5truecm}\quad
 + \frac{C}{t_{0}^{\frac{\gamma}{4}-\nu}} {\| f\|}^{2}_{L^{\infty}} \,
 {\| u_{+}\|}_{L^2(|x|^{\gamma})},
\end{eqnarray}
for any $0< \gamma<1$, and $\nu\leq \gamma/4$.

We continue to analize the non-source terms in (\ref{b0}). To this
end, notice that for any $z\in Y$ the following inequalities hold
true
\begin {equation}
 \label{b23}
 {\| z \|}_{L^2}\leq \frac{{\| z\|}_Y}{t^{\nu}} \qquad {\hbox{and}}\qquad
 {\| z\|}_{L^4 ((t,\infty), L^{\infty}(\R))}\leq \frac{{\| z \|}_{Y}}{t^{\nu}}
 \qquad \forall\, t\geq t_0.
\end{equation}
Also, recall that $z_{+}(t,x)=e^{it\partial_x^2} u_{+}$ (see
(\ref{b1})), so that from the well-known inequalities for the
solution of the free Schr\"odinger equation in (\ref{dis}) and
(\ref{L2}), we have that
\begin{equation}
 \label{b24}
 {\| z_{+} \|}_{L^{\infty}}= {\| e^{it\partial_{x}^2} u_{+} \|}_{L^\infty}\leq C\, \frac{{\| u_{+}\|}_{L^1}}{\sqrt{t}}
 \quad {\hbox{and}}\quad
 {\| z_{+} \|}_{L^{2}}= {\| e^{it\partial_{x}^2} u_{+} \|}_{L^2}={\| u_{+}\|}_{L^2}.
\end{equation}
First,  using (\ref{stri-inh}) with exponents $(4,\infty)$ and
$(\infty,2)$, (\ref{b23}) and the fact that $v_f(t,x)$ is a bounded
function (notice that $v_f(t,x)=\bar f(x/\sqrt{t})$, and $f$ is
bounded by Proposition~\ref{f}), we obtain the following control for
the second integral term on the r.h.s in (\ref{b0})
\begin{eqnarray}
 \label{b25}
&&
{\left\|
 \int_{t}^{\infty} e^{i(t-\tau)\partial_x^2}\{
 2(|v_f|^2-|f|_{\infty}^{2}) z + v_f^2 e^{-i\alpha \log \tau} \bar z
 \}\, \frac{d\tau}{\tau}
\right\| }_{Y}\leq
   \nonumber \\
&&
\hspace{4mm}
C\, \sup_{t\geq t_0} t^{\nu} \int_{t}^{\infty}
{\| 2(|v_f|^2-|f|^2_{\infty})z +v_f^2 e^{-i\alpha \log \tau} \bar z   \|}_{L^2}\, \frac{d\tau}{\tau}\leq
   \nonumber \\
&&
\hspace{4mm}
C\, {\| f\|}_{L^\infty}^{2} \sup_{t\geq t_0} t^{\nu} \int_{t}^{\infty} {\| z\|}_{L^2}\, \frac{d\tau}{\tau} \leq
   \nonumber \\
&&
\hspace{4mm}
C\, {\| f\|}_{L^\infty}^{2} {\| z\|}_{Y} \sup_{t\geq t_0} t^{\nu} \int_{t}^{\infty} \, \frac{d\tau}{\tau^{1+\nu}}
=
C\, {\| f\|}_{L^\infty}^{2} {\| z\|}_{Y},
\end{eqnarray}
for any $\nu\geq 0$.

Only the Duhamel term  in (\ref{b0}) related to the non-linear terms
$NLT(z+z_{+})$, where
$$
 NLT(z+z_{+})=
 2v_f e^{-i\frac{\alpha}{2}\log t} |z+z_{+}|^2 +
 \bar v_f e^{i\frac{\alpha}{2}\log t}(z+z_{+})^2+|z+z_{+}|^2(z+z_{+}).
$$
(see (\ref{NLTb})) remains to be estimated.

To control the terms associated to quadratic powers of  $z+z_{+}$,
we use as before the inequalities  (\ref{stri-inh}) with exponents
$(4,\infty)$ and $(\infty,2)$, the fact that ${\| v_f
\|}_{L^\infty}={\| f \|}_{L^\infty}<\infty$, and estimates
(\ref{b23}) and (\ref{b24}) to obtain that
\begin{eqnarray}
 \label{b26}
 &&
 {\left\|  \int_{t}^{\infty}e^{i(t-\tau)\partial_x^2}\{ 2v_fe^{-i\frac{\alpha}{2}\log \tau} |z+z_{+}|^2
 + \bar v_f e^{i\frac{\alpha}{2}} (z+z_{+})^2    \}\, \frac{d\tau}{\tau}   \right\|}_{Y}
    \nonumber \\
 &&
 \hspace{4mm}
 \leq C\, \sup_{t\geq t_0} t^{\nu}\int_{t}^{\infty}
 {\|
 2v_fe^{-i\frac{\alpha}{2}\log \tau} |z+z_{+}|^2
 + \bar v_f e^{i\frac{\alpha}{2}} (z+z_{+})^2
 \|}_{L^2}\, \frac{d\tau}{\tau}
       \nonumber \\
 &&
 \hspace{4mm}
 \leq C\,  {\| f\|}_{L^\infty} \sup_{t\geq t_0} t^{\nu}
 \int_{t}^{\infty} \{
 {\| z \|}_{L^2}({\| z \|}_{L^\infty}+ {\| z_{+} \|}_{L^\infty})
 +
{\| z_{+} \|}_{L^\infty} {\| z_{+} \|}_{L^2}
 \}\, \frac{d\tau}{\tau}
     \nonumber  \\
 &&
 \hspace{4mm}
 \leq
 C\, {\| f \|}_{L^\infty} {\| z \|}_{Y}\left(
 \sup_{t\geq t_0} t^{\nu}\int_{t}^{\infty} ({\| z \|}_{L^\infty}+\frac{{\| u_{+}\|}_{L^1}}{\sqrt{\tau}})\, \frac{d\tau}{\tau^{1+\nu}}
 \right) +
      \nonumber \\
 &&
 \hspace{8mm}
 C\, {\| f\|}_{L^\infty} {\| u_{+} \|}_{L^1} {\| u_{+} \|}_{L^2}\sup_{t\geq t_0} t^{\nu} \int_{t}^{\infty} \frac{d\tau}{\tau^{3/2}}
      \nonumber \\
 &&
 \hspace{4mm}
 \leq
 C\, {\| f \|}_{L^\infty} {\| z \|}_{Y}\left(  \frac{{\| z \|}_{Y}}{t_0^{\frac{1}{4}+\nu}}
 +\frac{{\| u_{+} \|}_{L^1}}{t_0^{\frac{1}{2}}} \right)+
 C\, {\| f \|}_{L^\infty} {\| u_{+} \|}_{L^1} {\| u_{+}\|}_{L^2}\frac{1}{t_{0}^{\frac{1}{2}-\nu}},
\end{eqnarray}
for all $0\leq \nu\leq 1/2$. Here, we have also used H\"older's
inequality in the $\tau$-variable to obtain the last inequality.

Next, notice that a straightforward computation  gives
$$
 |z+z_{+}|^2(z+z_{+})= |z|^2z+|z|^2z_{+}+ z^2 \bar z_{+}+ |z|^2 z_{+}+\bar z z_{+}^2 + z |z_{+}|^2 + z_{+}|z_{+}|^2.
$$
Then, similar arguments to the ones given to control the quadratic
terms in $z+z_{+}$ (that is, using (\ref{stri-inh}), pulling out of
the $L^2$-norm ${\| z\|}_{L^\infty}$ or ${\| z_{+} \|}_{L^\infty}$,
and using the estimates (\ref{b23}) and (\ref{b24})) give the
following control of $Y$-norm of the term associated to the cubic
term $|z+z_{+}|^2(z+z_{+})$
\begin{eqnarray}
 \label{b27}
  &&
 {\left\|  \int_{t}^{\infty}e^{i(t-\tau)\partial_x^2} (|z+z_{+}|^2(z+z_{+})) \, \frac{d\tau}{\tau}  \right\|
 }_{Y}
    \nonumber \\
 &&
 \hspace{4mm}
 \leq C\, \sup_{t\geq t_0} t^{\nu}\int_{t}^{\infty}
 {\|
 |z+z_{+}|^2(z+z_{+})
 \|}_{L^2}\, \frac{d\tau}{\tau}
       \nonumber \\
 &&
 \hspace{4mm}
 \leq C\, \sup_{t\geq t_0} t^{\nu}
 \int_{t}^{\infty} \left(
 {\| z \|}_{L^\infty}^{2}{\| z \|}_{L^2}
 + 3{\| z_{+} \|}_{L^\infty}{\| z \|}_{L^\infty} {\| z \|}_{L^2}+
 \right.
    \nonumber  \\
 &&
 \hspace{8mm}
\left.
 2{\| z_{+} \|}_{L^\infty}^{2} {\| z \|}_{L^2}
 + {\| z_{+} \|}_{L^\infty}^{2}{\| z_{+} \|}_{L^2}
 \right)\, \frac{d\tau}{\tau}
     \nonumber  \\
 &&
 \hspace{4mm}
 \leq
 C {\| z\|}_{Y}   \sup_{t\geq t_{0}} t^{\nu} \int_{t}^{\infty}
 \left(
 \frac{{\| z\|}_{L^\infty}^2}{\tau^{1+\nu}}
 + 3 {\| u_{+}\|}_{L^1} \frac{{\| z \|}_{L^\infty}}{\tau^{\frac{3}{2}+\nu}}
 + 2 \frac{{\| u_{+}\|}_{L^1}^{2}}{\tau^{2+\nu}}+
 \right)\, d\tau +
     \nonumber  \\
 &&
 \hspace{8mm}
C {\| u_{+}\|}_{L^1}^{2} {\| u_{+}\|}_{L^2}      \sup_{t\geq t_0} t^{\nu} \int_{t}^{\infty}\frac{d\tau}{\tau^2}
    \nonumber  \\
 &&
 \hspace{4mm}
 \leq
 C\frac{{\| z\|}_{Y}^{3}}{t_{0}^{\frac{1}{2}+2\nu}}+
 {\| u_{+}\|}_{L^1} {\| z\|}_{Y}^{2}\frac{1}{t_{0}^{\frac{3}{4}+\nu}} +
 C\, {\| u_{+}\|}_{L^1}^{2} {\| z\|}_{Y} \frac{1}{t_{0}} +
   \nonumber   \\
 &&
 \hspace{8mm}
 C{\| u_{+}\|}_{L^1}^2 {\| u_{+}\|}_{L^2}\frac{1}{t_{0}^{1-\nu}}
\end{eqnarray}
for all $0\leq \nu\leq 1$.
%
%----------------- Conclusion --------------------------------
%

Therefore, in view of the identity (\ref{b0}), and the inequalities
(\ref{b2}), (\ref{b22}) and (\ref{b25})-(\ref{b27}), we have that
\begin{eqnarray}
 \label{b28}
 {\| Bz \|}_{Y}
 &\leq&
 \frac{c(u_{+})}{t_0^{\frac{1}{4}-\nu} }\left( {\| |f(\cdot)|^2 - |f|^2_{\infty} \|}_{L^2}+
 {\| (\bar f)^2(\cdot)- (\bar f)^2_{\infty}(\cdot)   \|}_{L^2}
 \right)
   \nonumber   \\
 &+&
 \frac{c(u_{+})}{  t_0^{\frac{\gamma}{4}-\nu}  } {\| f \|}^2_{L^\infty} +
 \frac{c(u_{+})}{t_0^{\frac{1}{2}-\nu}} {\| f \|}_{L^\infty}+
 \frac{c(u_{+})}{t_0^{1-\nu}}
   \nonumber   \\
 &+&
 C\, {\| z \|}_{Y}
 \left\{
 {\| f \|}_{\infty}^2 +
 \frac{ {\| f\|}_{L^\infty} }{t_0^{\frac{1}{4}+\nu}} {\| z\|}_{Y} +
 \frac{c(u_{+})}{t_0^{\frac{1}{2}}} {\| f\|}_{L^\infty}
 \right.
    \nonumber   \\
 &&
 \left.
 \frac{{\| z\|}_{Y}^2}{t_0^{\frac{1}{2}+2\nu}}+
 \frac{c(u_{+})}{t_0^{\frac{3}{4}+\nu}} {\| z\|}_{Y}+
 \frac{c(u_{+})}{t_0}
 \right\}
\end{eqnarray}
for any $0\leq \nu\leq \gamma/4$. Here $c(u_{+})$ denotes a positive
constant which depends on the norm of $u_{+}$ in $L^1\cap
L^2(\langle x \rangle^{\gamma})$.

For any fixed $t_0>0$, and $0<\gamma<1$, by choosing $\nu=\gamma/4$,
from (\ref{b28}), we conclude that there exists a (small) positive
consatnt $B_0$, and a constant $R>0$ small with respect to $B_0$ and
$t_0$, such that for all $f$ solution of
$$
  f''+i \frac{s}{2} f'+\frac{f}{2} (|f|^2-A)=0
$$
satisfying $|f|_{+\infty}=|f|_{-\infty}$ and such that ${\|
f\|}_{L^\infty}\leq B_{0}$, and all $u_{+}$ small in $L^1\cap
L^2(\langle x \rangle^{\gamma})$ w.r.t. $t_0$,  $B_0$, ${\|
|f(\cdot)|^2 - |f|^{2}_{\infty}}_{L^2}$,  ${\| (\bar f)^2(\cdot)-
(\bar f)^{2}_{\infty}(\cdot)   \|}_{L^2}$ and $R$, the operator $B$
maps $\mathcal{B}_{R}$ into $\mathcal{B}_{R}$. On the other hand, by
bearing in mind that
$$
 {\| z+z_{+}  \|}_{L^\infty}\leq {\| z\|}_{L^\infty} + C \frac{{\| u_{+}\|}_{L^1}}{\sqrt{t}},
 \qquad {\hbox {for all}}\qquad z\in Y,
$$
similar arguments to the ones given in obtaining the estimates
(\ref{b25})-(\ref{b27}) shows that the operator $B$ defined by
(\ref{b0}) is a contraction on $(\mathcal{B}_R, {\|  \cdot
\|}_{Y})$. As a consequence, the application of the contraction
mapping principle yields the existence of a unique solution $z$ of
the equation (\ref{eqz}) such that
$$
 z\in \mathcal{C}([t_0,\infty), L^2(\R))\cap L^4([t_0,\infty), L^\infty(\R))
$$
satisfying
$$
 {\| z(t) \|}_{L^2(\R)}+ {\| z\|}_{L^4((t,\infty), L^\infty(\R))}=
 \mathcal{O}\left( \frac{1}{t^{\frac{\gamma}{4}}}  \right)
$$
as $t\rightarrow \infty$, for $0<\gamma<1$.

Performing the change of variables (\ref{change2}) and
(\ref{change1}), that is the changes defined by
$$
 v(t,x)= v_f(t,x)+ e^{i\frac{\alpha}{2}\log t} u
 \qquad {\hbox{with}}\qquad
 v_f(t,x)= \bar f \left(  \frac{x}{\sqrt{t}}\right),
 \qquad
 \alpha= 2|f|_{\infty}^{2}-A,
$$
and
$$
 u(t,x)= z(t,x)+ z_{+}(t,x), \qquad {\hbox{with}}\qquad
 z_{+}(t,x)=e^{it\partial_x^2}u_{+}
$$
gives the existence of a unique solution of (\ref{T1a}) such that
$$
 v-\tilde v_f\in \mathcal{C}([t_0,\infty), L^\infty(\R))\cap L^4([t_0,\infty), L^\infty(\R)),
$$
and satisfying (\ref{T1b}). To this end, since $\tilde v_f$ is
defined by (\ref{vftilde}), suffices to notice that
\begin{eqnarray*}
 v-\tilde v_f
 &=&
 ( v_f + e^{i\frac{\alpha}{2}\log t}u)-(v_f + e^{i\frac{\alpha}{2}\log t}(e^{it\partial_x^2}u_{+})(x))
    \\
 &=&
 e^{i\frac{\alpha}{2}\log t}(u- e^{it\partial_x^2} u_{+})
 =
 e^{i\frac{\alpha}{2}\log t} z
\end{eqnarray*}
so that
$$
 {\| v-\tilde v_f\|}_{L^2}+{\| v-\tilde v_f\|}_{L^4((t,\infty), L^2(\R))}=
 {\| z\|}_{L^2}+ {\| z\|}_{L^4((t,\infty), L^2(\R))}.
$$
\noindent Finally, we have to prove that under if the asymptotic
state $u_{+}$ satisfies that both $u_{+}$ and $\partial_x u_{+}$ are
in $L^1\cap L^2(\langle x \rangle^{\gamma})$, then the solution $v$
is such that  $v- \tilde v_f\in H^1$ and (\ref{T1c}) holds.

 Recall that solutions of (\ref{eqvF}) are in
correspondence with solutions $z$ of (\ref{eqz}) through the changes
of variables (\ref{change1})  and (\ref{change2})
(see~subsection~\ref{Preliminaries}). Define the auxiliary functions
$y=\partial_x u$ and $y_{+}=\partial_x z_{+}$, where as before
$z_{+}=e^{it\partial_x^2}u_{+}$. Then, if $z$ is a solution of
(\ref{eqz}), we have that $y$ has to be a solution of
\begin{eqnarray*}
 iy_{t}+ y_{xx}
 &=&
 \frac{1}{2t}\left(
 F_0(y_{+})- F_1(y)+ 2\partial_x(|v_f|^2) (z_{+}-z)+
 \partial_x(v_f^2) (\overline{z_{+}-z})e^{-i\alpha \log t}
 \right.
   \\
 &&
 \left.
 -\partial_{x} NLT(z+z_{+})
 \right).
\end{eqnarray*}
Now, notice that from the fact that $v_f(t,x)=\bar f(x/\sqrt t)$,
the properties of $f$ and $f'$ given in Proposition~\ref{f}, and
those of $z$ already proved, we conclude that the term
$$
 2\partial_x(|v_f|^2) (z_{+}-z)+
 \partial_x(v_f^2) (\overline{z_{+}-z})e^{-i\alpha \log t}
$$
is an integrable in time forcing term. As a consequence we can
follow the same argument as the one used to solve the equation for
$z$, and concude that (\ref{T1c}) holds. This finishes the proof.
%
%--------------------------------------------------------
\subsection{Proof of Theorem~\ref{T2}}
%--------------------------------------------------------
%
%
Let $\tilde t_0>0$ and $0<\gamma<1$. Define $t_0=\frac{1}{\tilde
t_0}$, and denote by $v$ the associated solution of the equation
(\ref{T1a}) verifying
\begin{equation}
 \label{c1}
 {\| v-\tilde v_f \|}_{L^2(\mathbb{R})}+
 {\| v-\tilde v_f \|}_{L^4((t,\infty), L^\infty(\mathbb{R}))}=
 \mathcal{O}\left(  \frac{1}{t^{\frac{\gamma}{4}}} \right),
 \qquad {\hbox{as}}\quad t\rightarrow \infty,
\end{equation}
given by Theorem~\ref{T1}.

Define $u$ to be the pseudo-conformal transformation of the solution
$v$, i.e.
$$
 u=\mathcal{T}(v).
$$
Then, $u$ satisfies equation (\ref{eqs1}). Next, notice that (recall
the definition of $\tilde u_f$ and $\mathcal{T}$ in (\ref{u1}) and
(\ref{v}), respectively)
$$
 \tilde u_f:= \mathcal{T}\left(
 v_f+ (2\pi) e^{i\frac{\alpha}{2} \log t} \frac{e^{i\frac{x^2}{4t}}}{\sqrt{4\pi i t}}\, \hat u_+ \left(\frac{x}{2t}  \right)
 \right),
\qquad \alpha=2|f|^2_{\infty}-A,
$$
so that
$$
 u-\tilde u_f= \mathcal{T}\left(
 v-\left(  v_f+ (2\pi) e^{i\frac{\alpha}{2} \log t} \frac{e^{i\frac{x^2}{4t}}}{\sqrt{4\pi i t}}\, \hat u_+  \left(\frac{x}{2t} \right) \right)
 \right)
$$
where
\begin{multline}
 v-\left(
 v_f+ (2\pi) e^{i\frac{\alpha}{2} \log t} \frac{e^{i\frac{x^2}{4t}}}{\sqrt{4\pi i t}} \,
 \hat u_+ \left(\frac{x}{2t} \right)
 \right)=
  \nonumber \\
 (v-\tilde v_f) +  e^{i\frac{\alpha}{2}\log t} \left(
  \left( e^{it\partial^2_x} u_+ \right) (x)-
  (2\pi)  \frac{e^{i\frac{x^2}{4t}}}{\sqrt{4\pi i t}}\,
 \hat u_+ \left(\frac{x}{2t} \right)
 \right),
\end{multline}
recall the definition of $\tilde v_f$ in (\ref{tildevf}).

Due to the invariance of $L^2(\R)$ and $L^4\left(  (0,t),
L^\infty(\R) \right)$ under the pseudo-conformal transformation
$\mathcal{T}$, and the decay estimates (\ref{c1}), in order to prove
(\ref{C1a}) it suffices to study the behaviour of
$$
 \left( e^{it\partial_x^2} u_+ \right)(x)-
 (2\pi)  \frac{e^{i\frac{x^2}{4t}}}{\sqrt{4\pi i t}} \, \hat u_+
 \left(  \frac{x}{2t}\right)
$$
in $L^2(\R)$ and $L^4\left( (t,\infty), L^\infty(\R)  \right)$, as
$t$ goes to infinity.

On the one hand, using the expression of $e^{it\partial_x^2}u_+$ as
a convolution (see (\ref{sh-solution1})) and Plancherel's indentity,
we have
\begin{eqnarray*}
 &&
 \left\|
 \, \left( e^{it\partial_x^2} u_+ \right)(x)-
 (2\pi)  \frac{e^{i\frac{x^2}{4t}}}{\sqrt{4\pi i t}} \, \hat u_+
 \left(  \frac{x}{2t}\right)
 \right\|_{L^2}=
   \nonumber \\
 &&
 \hspace{1truecm}
  C\, {\| (u_+(\cdot) (e^{i\frac{(\cdot)^2}{4t}}-1))^{\widehat{\ }}  \|}_{L^2}
 =C{\| u_+(\cdot) (e^{i\frac{(\cdot)^2}{4t}}-1) \|}_{L^2}
 \leq \frac{C}{t^{\frac{\gamma}{4}}}\, {\| u_+\|}_{L^2(|x|^\gamma)}.
\end{eqnarray*}
Here, we have used Lemma~\ref{lemma1} in obtaining the last
inequality.

On the other hand, from the decay estimate (\ref{dis})
$$
  {\|  e^{it\partial_x^2} u_+ \|}_{L^4\left(  (t,\infty), L^\infty(\R) \right)}
  \leq \frac{C}{t^{\frac{1}{4}}}{\| u_+\|}_{L^1}
$$
from which it follows that
\begin{eqnarray*}
 &&
 {\left\|
  \left(  e^{it\partial_x^2} u_+ \right)(x)-
  (2\pi) \frac{e^{i\frac{x^2}{4t}}}{\sqrt{4\pi i t}}\, \hat u_+\left( \frac{x}{2t} \right)
 \right\|}_{L^4\left( (t,\infty), L^\infty(\R) \right)}  \leq
   \\
 &&
 \hspace{1truecm}
 \frac{C}{t^{\frac{1}{4}}}\, {\| u_+\|}_{L^1}+
 \frac{C}{t^{\frac{1}{4}}}\, {\| \hat u_+\|}_{L^\infty}
 =\mathcal{O}\left(  \frac{1}{t^{\frac{1}{4}}} \right),
 \qquad {\hbox{as}}\quad t\rightarrow \infty.
\end{eqnarray*}
From the above inequalities, we get
\begin{eqnarray*}
  &&
 {\left\|
  \left(  e^{it\partial_x^2} u_+ \right)(x)-
  (2\pi) \frac{e^{i\frac{x^2}{4t}}}{\sqrt{4\pi i t}}\, \hat u_+\left( \frac{x}{2t} \right)
 \right\|}_{L^2(\R)}  +
    \\
 &&
 {\left\|
  \left(  e^{it\partial_x^2} u_+ \right)(x)-
  (2\pi) \frac{e^{i\frac{x^2}{4t}}}{\sqrt{4\pi i t}}\, \hat u_+\left( \frac{x}{2t} \right)
 \right\|}_{L^4\left( (t,\infty), L^\infty(\R) \right)}
 =
 \mathcal{O}\left(  \frac{1}{t^{\frac{\gamma}{4}}}\right),
 \qquad {\hbox{as}}\quad t\rightarrow \infty,
\end{eqnarray*}
for any $0<\gamma< 1$, and $u_+\in L^1\cap L^2(\langle x
\rangle^\gamma)$.

Now, (\ref{C1b}) is an immediate consequence of the triangle
inequality, (\ref{C1a}) and Plancherel's identity. Also, inequality
(\ref{C1c}) follows from (\ref{C1a}) and (\ref{C1b}), by using the
general inequality
$$
 {\| |f|^2-|g|^2\|}_{L^2}\leq (  {\| f\|}_{L^2} +{\| g \|}_{L^2})\, {\| f-g\|}_{L^2},
$$
for any functions $f$ and $g$ in $L^2$.

Finally, assume by contradiction that there exists $g(\cdot,t)\in
L^2(\R)$ defined in a time interval $(0,T_0>0]$ such that
\begin{equation}
 \label{c2}
 {\left\|  u(t,x) - \frac{e^{i\frac{x^2}{4t}}}{\sqrt{t}}\, f\left(  \frac{x}{\sqrt{t}} \right)-g(t,x)  \right\|}_{L^2}\rightarrow 0
 \qquad {\hbox{as}}\qquad t\rightarrow 0.
\end{equation}
Then, using the definition of $\tilde u_f$ in (\ref{u1}) and the
triangular inequality, we obtain that
\begin{eqnarray*}
 &&{\left\|
 \sqrt{\pi i}  e^{i\frac{\alpha}{2} \log t} \widehat{\overline{u_{+}}}  \left( -\frac{x}{2}\right)-g(t,x)  \right\|}_{L^2}
    \nonumber \\
 &&
 \hspace{1truecm} =
 {\left\|  \tilde u_f(t,x) - \frac{e^{i\frac{x^2}{4t}}}{\sqrt{t}}\, f\left(  \frac{x}{\sqrt{t}} \right)-g(t,x)  \right\|}_{L^2}
    \nonumber  \\
 &&
 \hspace{1truecm} =
 {\left\| (\tilde u_f -u)(t,x) +\left(
  u(t,x)-  \frac{e^{i\frac{x^2}{4t}}}{\sqrt{t}}\, f\left(  \frac{x}{\sqrt{t}} \right)-g(t,x)   \right)
 \right\|}_{L^2}
    \nonumber  \\
 &&
 \hspace{1truecm}  \leq
 {\| \tilde u_f -u  \|}_{L^2} +
 {\left\| u(t,x)-  \frac{e^{i\frac{x^2}{4t}}}{\sqrt{t}}\, f\left(  \frac{x}{\sqrt{t}} \right)-g(t,x)     \right\|}_{L^2}.
\end{eqnarray*}
Thus, from (\ref{C1a}) and (\ref{c2}) and the above identity we
conclude that
$$
   g(t,x)= \sqrt{\pi i} e^{i\frac{\alpha}{2}\log t}  \widehat{\overline{u_{+}}}  \left( -\frac{x}{2}\right)
   \qquad {\hbox{a.e}}\quad x,
$$
which does not have limit in $L^2$ as $t\rightarrow 0$, unless
$\alpha= 2|f|^2_{\infty}-A =0$.

It remains to prove (\ref{C1d1}) and (\ref{C1d2}). Using (\ref{T1c})
and the inequality $|g|^2\leq {\| g \|}_{L^2} {\| g \|}_{L^2}$ in
one dimension, we get that
\begin{equation}
 \label{c3}
 {\| v-\tilde v_f  \|}_{L^{\infty}}= \mathcal{O}\left(  \frac{1}{t^{\frac{\gamma}{4}}} \right), \qquad t\rightarrow \infty, \qquad (0<\gamma<1).
\end{equation}
From the definition of $u$ in terms of $v$, given by the
pseudo-conformal transformation (\ref{v}), we write
\begin{eqnarray*}
 u(t,x)
 &=&
 \mathcal{T} v(t,x)= \frac{e^{i\frac{x^2}{4t}}}{\sqrt{t}} \bar v \left( \frac{1}{t}, \frac{x}{t} \right)
   \\
 &=&
 \frac{e^{i\frac{x^2}{4t}}}{\sqrt{t}} \left(
 {\overline{v- \tilde v_f}} \left( \frac{1}{t}, \frac{x}{t} \right)-
  {\overline{\tilde v_f}} \left( \frac{1}{t}, \frac{x}{t} \right)
 \right),
\end{eqnarray*}
where (see (\ref{vftilde}) or (\ref{tildevf}))
$$
 \tilde v_f(t,x)=
 \bar f\left( \frac{x}{\sqrt{t}}\right)+
 e^{i\frac{\alpha}{2}\log t} \left( e^{it\partial_x^2} u_{+}  \right)(x), \qquad
 \alpha= 2|f|_{\infty}^{2}-A.
$$
Then (\ref{C1d1})-(\ref{C1d2}) follow from the above identities,
(\ref{c3}), and the decay estimate for $e^{it\partial_x^2}u_{+}$
given in (\ref{dis}). This finishes the proof of Theorem~\ref{T2}.

%-----------------------------------------------------------------------------------
\subsection{Proof of Corollary~\ref{T3}}
Theorem~\ref{T2} gives the existence of a filament function $u(t,x)$
which  is regular and bounded for $0<t<\tilde t_0$. From the
filament function $u$ given by Theorem~\ref{T2}, one can construct a
corresponding curve $\Xn$ solution of LIE.

First, notice that at least in the case of odd solutions the curve
$\Xn_{f}(t,x)$ has a point of curvature $0$ (the curvature of an odd
solution vanishes at least at the point $x=0$), and as a consequence
here we need to consider a different parallel frame (other that the
Serret-Frenet frame) to avoid the restriction that the curvature of
the curvature should not vanish. Precisely, one can consider the
parallel frame of vectors  $\{ \Tn, \mathbf{e_1}, \mathbf{e_2}\}$
given by the system of equations
\begin{eqnarray}
 \label{T31}
 \left\{
 \begin{array}{ll}
 \Tn_x=\qquad \alpha \mathbf{e_1}+ \beta \mathbf{e_2} \\[2ex]
 \mathbf{e_1}_x= -\alpha \Tn\\[2ex]
 \mathbf{e_2}_x= -\beta \Tn,
 \end{array}
 \right.
\end{eqnarray}
where the quantities $\alpha$ and $\beta$ are defined through the
function $u$ by
%\begin{equation}
% \label{T32}
$$
 u=\alpha+ i\beta,
$$
%\end{equation}
to construct the tangent vector $\Tn$ solution of $\Tn_t=\Tn\times
\Tn_{xx}$. Then, using the regularity of $u$, and after integration
with the initial conditions
$$
 \Xn(\tilde t_0, 0)=(0,0,0)
 \qquad {\hbox{and}}\qquad
 \Xn_{x}(\tilde t_0,0)= (1,0,0),
$$
we get a curve $\Xn(t,x)$ solution of LIE.\footnote{ Conversely,
using the parallel frame defined by the system (\ref{T31}), it can
be also proved that if $\Xn(t,x)$ is a regular solution of LIE, and
define the function $u=\alpha+i \beta$, then $u$ solves the
$1$d-cubic Schr\"odinger equation
$$
\displaystyle{
 iu_t+ u_{xx}+\frac{u}{2} (|u|^2-A(t))=0
}
$$
with $A(t)= -|u|^2(0,t)/2- < \partial_t\mathbf{e_{1}},
{\mathbf{e_2}}>(0,t)$. } The details can be found for example in
\cite{BV3} and \cite{BV4}, see also \cite{NSVZ}.

Once $\Xn(t,x)$ has been constructed for $0<t<\tilde t_0$, part
{\it{i)}} is an immediate consequence of (\ref{C1d1})-(\ref{C1d2}),
the boundedness property of $f$, and the fact that $u$ is the
filament function associated to $\Xn(t,x)$ (thus
$|u(t,x)|$=$|c(t,x)|$, with $c$ the curvature of $\Xn$).

The existence of $\X_0(x)$, the trace of $\Xn(t,x)$ at time $t=0$,
follows from the integrability of $\Xn_t$ at $t=0$ thanks to the
uniform bound of the curvature in part {\it{i)}}. Indeed, since
$\Xn(t,x)$ is a solution of LIA, from the system of equations
(\ref{T31}), and the fact that the vectors $\mathbf{e_1}$ and
$\mathbf{e_2}$ are unitary, it follows that
\begin{eqnarray*}
 |\Xn_t(t,x)|
 &=&
 |\Xn_x\times \Xn_{xx}|=|\Tn\times \Tn_{x}|=
 |\Tn\times (\alpha \mathbf{e_1}+\beta \mathbf{e_2})|
   \\
 &=&
 |\alpha \mathbf{e_2}-\beta \mathbf{e_1}|=
 \sqrt{\alpha^2+\beta^2}=
 |u(t,x)|=|c(t,x)|
 \leq \frac{c_1}{\sqrt{t}},
\end{eqnarray*}
uniformly on the interval $x\in(-\infty, \infty)$, since
$u=\alpha+i\beta$.

Therefore, for any fixed positive times $t_1$, and $t_2$ with
$t_1<t_2$, we have that
\begin{eqnarray*}
 |\Xn(t_1,x)- \Xn(t_2,x)|
 &=&
 \left|
 \int_{t_1}^{t_2} \Xn_t(t', x)\, dt'
 \right|
 \leq
 \int_{t_1}^{t_2}
 |\Xn_t(t',x)|\, dt'
  \\
 &\leq&
 c_1\,
 \int_{t_1}^{t_2}
 \frac{dt'}{\sqrt{t'}},
\end{eqnarray*}
from which the existence of the limit $\lim_{t\rightarrow 0}
\Xn(t,x)=\Xn_{0}(x)$ follows by taking $t_2=t>0$ and letting $t_1$ go to zero
in the above inequality. Moreover, we have
$$
 |\Xn(t,x)-\Xn_0(x)|\leq 2c_1\sqrt{t}.
$$
Finally, the regularity property of  $\Xn_{0}$ easily follows
from the above inequality, and the identity
$$
 \Xn_0(x)-\Xn_0(y)=
 [\Xn_0(x)- \Xn(t,x)] -
 [\Xn_0(y)-\Xn(t, y)] +
 [\Xn(t,x)- \Xn(t,y)].
$$
To this end, if suffices to observe that
$$
 |\Xn(t,x)- \Xn(t,y)|=
 \left|
 \int_{x}^{y}
 \Tn(t,z)\, dz
 \right| \leq |x-y|
$$
since the tangent vector to the curve, $\Tn$, is unitary. As a
consequence,
$$
 |\Xn_0(x)- \Xn_{0}(y)|\leq 2c_1\sqrt{t}+ |x-y|\leq c_3 |x-y|
$$
for some non-negative constant $c_3$, whenever $t$ is sufficiently
small. Therefore, we conclude that $\Xn_0(x)$ is a Lipschitz
continuous function.
%
%%%%%%%%%%%%%%%%%%%%%%%%%%%%%%%%%%%%%%%%%%%%%%%%%%%%%%%%%%%%%%%%
\section{The initial value problem for the principal value distribution}
\label{section-pv} We begin this section proving the existence of
non-trivial solutions
$$
 u_f(t,x)= \frac{e^{i\frac{x^2}{4t}}}{\sqrt{t}}f \left( \frac{x}{\sqrt{t}} \right)
$$
of
$$
 iu_t +u_{xx}+\frac{u}{2} (|u|^2-\frac{A}{t})=0
$$
such that $u_f(t,\cdot)$ converges as a distribution to
$$
 u_f(0,x)= z_0\, \pv \frac{1}{x},
$$
for some $z_0\in \C\setminus \{0\}$, and appropriate values of $A$.
Moreover, these solutions are characterized by the property that
$2|f|^2_{\infty}-A=0$, so that the solution $u$ constructed in
Theorem~\ref{T2} has a trace at $t=0$. We have the following lemma:
\begin{lemma}
\label{lemma-pv}
 For any $a\neq 0$, there exist $A_a$ and a non-trivial odd solution $f$ of
 \begin{equation}
 \label{lemma-f}
 f''+i\frac{x}{2} f'+\frac{f}{2} (|f|^2-\frac{A_a}{t})=0,
 \end{equation}
such that
$$
 \lim_{t\longrightarrow 0^{+}} \frac{e^{i\frac{x^2}{4t}}}{\sqrt{t}}f\left( \frac{x}{\sqrt{t}}  \right)= z_0 \pv \frac{1}{x}, \quad z_0\neq 0
$$
in the distributional sense. Moreover,
$$
|z_0|=2|f'|_{\infty}
\qquad {\hbox{with}}\qquad
 \frac{\sqrt{3}}{2} |a| \leq |z_0|< |a|.
$$
In addition, $f$ satisfies ${\|f \|}_{L^\infty}\leq 2|a|$.
\end{lemma}
\begin{proof}
First, assume  $f$ is an odd solution of (\ref{lemma-f}) such that
$$
 2|f|^2_{\infty}-A=0.
$$
Then, from the asymptotic behaviour for odd solutions of (\ref{lemma-f}) established in Proposition~\ref{f}, it easily follows that
$$
 f(x)= |f|_{\infty} e^{ic_{+}} e^{i\phi_2(x)} \sgn (x)
     + 2i |f'|_{\infty} e^{id_{+}} \frac{e^{i\phi_3(x)}}{x}
     +\mathcal{O}\left( \frac{1}{|x|}  \right)
$$
as $|x|\longrightarrow \infty$, with

$$
  \phi_2(x)= (|f|^2_{\infty}-A)\log |x|
  \qquad {\hbox {and}}\qquad
  \phi_3(x)= -(x^2/4)- (2|f|^2_{\infty}-A)\log |x|,
$$
and $d_{+}\in [0,2\pi)$.

Recall that $f$ is regular and odd, then, by using the dominated convergence theorem, we have that
$$
 \frac{e^{i\frac{x^2}{4t}}}{\sqrt{t}} f\left( \frac{x}{\sqrt{t}}  \right)\chi_{|x|\leq M\sqrt{t}}(x), \qquad M>> 1
$$
goes to zero in $\mathcal{S}'(\R)$ as $t\rightarrow 0^{+}$. On the
other hand, the function $g(x)= e^{i(\frac{x^2}{4}+\phi_2(x))}\sgn
(x)$ is a bounded, odd and has a continuous Fourier transform that
is zero at zero. Hence, by Parseval theorem $\frac{1}{\sqrt{t}}
g(x/\sqrt{t})$ also tends to zero as $t\downarrow 0$. Finally, the
error term is integrable and odd, therefore arguing as we did before
the convergence of the error term to zero follows by using the
dominated convergence theorem.

The  convergence of
$$
 u_f(t,x)= \frac{e^{i\frac{x^2}{4t}}}{\sqrt{t}} f\left( \frac{x}{\sqrt{t}} \right)
$$
as $t\downarrow 0$ to $z_0\pv (1/x)$, with $z_0$ such that $|z_0|= 2|f'|_{\infty}$ easily follows from previous remarks and  the hypothesis that $2|f|^2_{\infty}-A=0$.

Now, we continue to prove that, for any given $a\neq 0$, there
exists $A_a\in \R$, and an odd solution $f$ of (\ref{lemma-f})
satisfying the condition $2|f|^2_{\infty}-A_a=0$.

Indeed, for fixed $a\neq 0$, and $-1\leq \lambda\leq 1$, let $\Xn_{a,\lambda}(t,x)$ be an odd solution of LIA, that is a solution of LIA the form $\Xn_{a,\lambda}(t,x)= e^{\frac{\A}{2}\log t}\sqrt{t} \Gn_{a,\lambda}(x\sqrt{t})$ with $\Gn_{a,\lambda}$ the solution of (\ref{G}) with the initial conditions
\begin{equation}
 \label{pv-1}
 \Gn_{a,\lambda}(0)=(0,0,0)
 \qquad {\hbox{and}}\qquad
 (\Gn_{a,\lambda})'(0)=(0,\sqrt{1-\lambda^2}, \lambda).
\end{equation}
Define the function $F_a$ as follows
$$
  F_a(\lambda)= 2 T_{3,a,\lambda}(\infty)- T_{3,a,\lambda}(0),
$$
where, as before $T_{3,a,\lambda}$ denotes the third component of the tangent vector to the curve $\Xn_{a,\lambda}$\footnote{Recall that for odd solutions of LIA, the third component of the associated tangent vector, $T_{3}$, is an even function. Thus, in particular $T_3(+\infty)=T_3(-\infty)$}.
Notice that, for $\lambda=1$, $\Gn_{a,1}(x)=(0,0,x)$, and therefore $F_a(1)=2- 1>0$. Also, for $\lambda=-1$, $\Gn_{a,-1}(x)=(0,0,-x)$, and therefore $F_a(-1)= -2-(-1)<0$. Since the map $(\Gn(0), \Gn'(0),a)\longrightarrow T_3(\infty)$ is continuous (see~\cite[Proposition~2, pp.~2101]{GV}), we conclude that there exists $\lambda_a\in (-1,1)$ such that $F_{a}(\lambda_a)=0$, that is $2T_{3,a,\lambda_{a}}-T_{3,a,\lambda_{a}}=0$.

Notice that the associated function $f$ (through the Hasimoto transfom  and the change of variables (\ref{u-f})) is an odd solution of
$$
 f''+i \frac{x}{2}f'+\frac{f}{2}(|f|^2-A_{a})=0
$$
with $A_a=a\lambda_a$ (recall that $A_{a}$ is given in terms of the initial conditions (\ref{pv-1}) by the identity (\ref{A1})), and from (\ref{key-identities}) we have that
\begin{eqnarray*}
 2|f|^2_{\infty}-A_a
 &=&
 2(-aT_{3,a,\lambda_a}(\infty)+A_a)-A_a= -2aT_{3,a,\lambda_a}+A_a
    \\
 &=&
 -a(2T_{3,a,\lambda_a}(\infty)- T_{3,a,\lambda_a(0)})=0,
\end{eqnarray*}
since $T_{3,a,\lambda_a}(0)= \lambda_a$ (see (\ref{pv-1})).

Finally, for odd solutions of LIA, notice that the conservation law in Proposition~\ref{f} becomes
$$
 |f'|^2(x)+\frac{1}{4}(|f|^2(x)-A_a)^2=\frac{a^2}{4}
$$
(the value of the constant on the r.h.s of the above identity
follows from the identities (\ref{key-identities}) and the initial
conditions (\ref{pv-1})). From which we get that
$$
 |f(x)|\leq |a|+|A_a|= |a|(1+|\lambda_a|), \qquad \forall\, x\in\R
$$
and
$$
 |f'|^2_{\infty}= \frac{a^2}{4}-\frac{1}{4}(|f|^2_{\infty}-A_a)^2=
 \frac{a^2}{4}-\frac{1}{4} \left(  \frac{A_a}{2} -A_a \right)^2=
 \frac{a^2}{4}\left( 1- \frac{\lambda_a^2}{4}  \right)
$$
by using the condition $2|f|^2_{\infty}-A_{a}=0$, and that
$A_a=a\lambda_a$. From the above formulae we conclude that
$$
 \frac{3}{16} a^2<|f'|^2_{\infty}\leq \frac{a^2}{4}
 \qquad {\hbox{and}}\qquad
 {\|f\|}_{L^\infty}\leq 2|a|
$$
since $\lambda_a\in (-1,1)$. This concludes the proof of the lemma.
\end{proof}
%%%%%%%%%%%%%%%%%%%%%%%%%%%%%%%%%%%%%%%%%%%%%%%%%
%
As a consequence of Theorem~\ref{T2} and Lemma~\ref{lemma-pv}, we obtain the following result:
%%%%%%%%%%%%%%%%%%%%%%%%%%%%%%%%%%%%%%%%%%%%%%%%%%%%%%%%%%%%%%%%%%%%%%
\begin{theorem}
 \label{T4}
 Let $a\neq 0$ sufficiently small, and consider $A_a$, $f$, and $z_0$ as in Lemma~\ref{lemma-pv}. Then,
 there exists $\varepsilon >0$ such that for any given $u_{+}$ with
 ${\|u_{+}\|}_{L^1\cap L^2(\langle x \rangle^{\gamma})}\leq \varepsilon$ and
 $0<\gamma<1$, the initial value problem:
 \begin{eqnarray*}
 \left\{
 \begin{array}{ll}
 \displaystyle{
 iu_t+u_{xx}+\frac{u}{2} (|u|^2- \frac{A_{a}}{t})=0
 }& \\[2ex]
 u(0,x)= z_0 \pv \frac{1}{x}+ \sqrt{\pi i} \widehat{\overline{u_{+}}}\left(-\frac{x}{2}\right)
 \end{array}
 \right.
 \end{eqnarray*}
 has a unique solution $u(t,x)$ such that
 $$
 u-\tilde u_f\in\mathcal{C}((0,1], L^2(\R))\cap L^4((0,1], L^{\infty}(\R))
 $$
 where
 $$
  \tilde u_f(t,x)= \frac{e^{i\frac{x^2}{4t}}}{\sqrt{t}}f\left( \frac{x}{\sqrt{t}} \right)+
  \sqrt{\pi i} \widehat{\overline{u_{+}}}(-x/2).
 $$
\end{theorem}
%%%%%%%%%%%%%%%%%%%%%%%%%%%%%%%%%%%%%%%%%%%%%%%%%%%%%%%%%%%%%%%%%%%%%%%
Theorem~\ref{T4} represents a well-posedness result for the initial
value problem
\begin{eqnarray}
 \label{pv}
 \left\{
 \begin{array}{ll}
 \displaystyle{
 iu_t+u_{xx}+\frac{u}{2}(|u|^2-\frac{A}{t})=0
 } &
   \\[2ex]
 u(0,x)=z_0 \pv \frac{1}{x},
 \end{array}
 \right.
\end{eqnarray}
for some values of $z_0$ and adequate constants $A$ in (\ref{pv}):
If we denote by $u_f(t,x)$ the solution of the IVP (\ref{pv}), we
have proved that there exist appropriate (small) perturbations $u$
of the solution $u_f$ such that
$$
 \lim_{t\rightarrow 0} u(t,x)= z_0 \pv \frac{1}{x}
 + \sqrt{\pi i} \widehat{\overline{u_{+}}}\left(-\frac{x}{2}\right).
$$
In particular, $u-u_f$ has a trace in $L^2$, i.e. there exists the
limit in $L^2$ of $u-u_f$ as $t\rightarrow 0^{+}$. This is in
contrast with the situation in which one considers as initial datum
the delta distribution. In the latter case, it was shown in
\cite{BV2} (see also \cite{BV1}) that when considering the IVP
\begin{eqnarray}
 \label{delta}
 \left\{
 \begin{array}{ll}
 \displaystyle{
 iu_t+u_{xx}+\frac{u}{2}(|u|^2-\frac{c_0^2}{t})=0
 } &
   \\[2ex]
 u(0,x)= \sqrt{4\pi i} \, c_0 \, \delta_{x=0},\qquad c_0\neq 0
 \end{array}
 \right.
 \end{eqnarray}
there exist (small) perturbations $u$ of  the solution
$u_{c_0}(t,x)=  c_0\frac{e^{i\frac{x^2}{4t}}}{\sqrt{t}} $ of the IVP
 (\ref{delta}) such that the limit of $u-u_{c_0}$ as $t\rightarrow
0^{+}$ does not exist in $L^2$. As a consequence the IVP  for the
Dirac-delta (\ref{delta}) is ill-posed.

%-------------------------------------------------------------------------
\section{ACKNOWLEDGMENTS}
 L.~Vega is funded in part by the grant MTM 2007-82186 of MEC
 (Spain) and FEDER.

Part of this work was done while the first author was visiting the
Universidad del Pa\'is Vasco under the PIC program. S.~Guti\'errez
was  partially supported by the grant MTM 2007-82186 of MEC (Spain).
Financial support from the program ``Euclidean Harmonic Analysis,
Nilpotent Lie Groups and PDEs",  held in the Centro di Ricerca
Matematica Ennio De Giorgi in Pisa, is also kindly acknowledged by
S.~Guti\'errez.

%-------------------------------------------------------------------------

%
%%%%%%%%%%%%%%%%%%%%%%%%%%%%%%%%%%%%%%%%%%%%%%%%%%%%%

%%%%%%%%%%%%%%%%%
\end{document}